\documentclass[a4paper, 11pt,DIV=12]{scrartcl}
\usepackage{amsfonts}
\usepackage{amsmath} 
\usepackage{amssymb}
\usepackage[dvipsnames,svgnames,table]{xcolor}
\usepackage{graphicx}
\usepackage{amsthm}
\usepackage{tikz}
\usetikzlibrary{calc}
\usetikzlibrary{intersections}
\usetikzlibrary{arrows.meta}
\usepackage[colorlinks=true,linkcolor=RoyalBlue,citecolor=PineGreen,urlcolor=RoyalBlue]{hyperref} 
\hypersetup{linktocpage=true}
\usepackage[nameinlink]{cleveref}
\usepackage{booktabs}
\usepackage[title]{appendix}
\usepackage{fontawesome5}

\crefname{equation}{}{}

\newtheorem{theorem}{Theorem}[section]
\newtheorem{lemma}[theorem]{Lemma}
\newtheorem{conjecture}[theorem]{Conjecture}
\newtheorem{proposition}[theorem]{Proposition}
\newtheorem{corollary}[theorem]{Corollary}
\theoremstyle{definition}\newtheorem{definition}[theorem]{Definition}
\theoremstyle{definition}\newtheorem{example}[theorem]{Example}
\theoremstyle{definition}\newtheorem{remark}[theorem]{Remark}
\theoremstyle{definition}

\DeclareMathOperator{\rank}{rank}

\newcommand{\diff}{\mathrm{d}}

\colorlet{colbg}{white}
\colorlet{colfg}{black}
\colorlet{colgraphv}{colfg!75!colbg}
\colorlet{colgraphe}{colfg!55!colbg}
\colorlet{colG}{DarkSeaGreen}
\definecolor{colR}{HTML}{CC6677}
\definecolor{colO}{HTML}{DDCC77}
\definecolor{colB}{HTML}{6699CC}
\colorlet{colY}{Gold!90!black}
\colorlet{colGray}{white!60!black}

\colorlet{col1}{colR}
\colorlet{col2}{colG}
\colorlet{col3}{colGray}
\colorlet{col4}{colO}
\colorlet{col5}{colB}
\colorlet{colr1}{colR}
\colorlet{colr2}{colB}
\colorlet{colr3}{colO}
\colorlet{colr4}{colG}
\colorlet{colr5}{colGray}
\colorlet{colr6}{Tan}
\colorlet{colr7}{Plum}
\colorlet{colr8}{LightSkyBlue}
\colorlet{colr9}{LightGreen}
\colorlet{cola}{colfg!75!colbg}

\tikzstyle{vertex}=[fill=colgraphv,circle,inner sep=0pt, minimum size=4pt]
\tikzstyle{edge}=[line width=1.5pt,colgraphe]

\tikzstyle{axes}=[draw=cola,-{Latex[round,width=3pt]}]
\tikzstyle{aline}=[draw=cola]
\tikzstyle{bline}=[draw=cola!10!colbg]
\tikzstyle{alabelsty}=[cola,font=\scriptsize]
\tikzstyle{labelsty}=[font=\scriptsize]
\tikzstyle{slabelsty}=[font=\tiny]

\newcommand{\bnr}[2]{\tikz[baseline=3pt]{\fill[#2] (0,0)rectangle(#1,0.45);\node[anchor=west] at (0,0.225) {#1};}}

\newcommand{\bnrsm}[3]{\tikz[baseline=3pt]{\fill[#3] (0,0)rectangle(#1*#2-#2,0.45);\node[anchor=west] at (0,0.225) {#1};}}

\begin{document}

\title{The number of realisations of a rigid graph in Euclidean and spherical geometries}

\date{October 11, 2023}

\author{Sean Dewar\thanks{
School of Mathematics, University of Bristol, Bristol, UK. E-mail: \texttt{sean.dewar@bristol.ac.uk}}\,  and
Georg Grasegger\thanks{
Johann Radon Institute for Computational and Applied Mathematics (RICAM), Austrian Academy of Sciences, Linz, Austria. E-mail: \texttt{georg.grasegger@ricam.oeaw.ac.at}}
}

\maketitle

\begin{abstract}
	A graph is \emph{$d$-rigid} if for any generic realisation of the graph in $\mathbb{R}^d$ (equivalently, the $d$-dimensional sphere $\mathbb{S}^d$), there are only finitely many non-congruent realisations in the same space with the same edge lengths.
	By extending this definition to complex realisations in a natural way,
	we define $c_d(G)$ to be the number of equivalent $d$-dimensional complex realisations of a $d$-rigid graph $G$ for a given generic realisation,
	and $c^*_d(G)$  to be the number of equivalent $d$-dimensional complex spherical realisations of $G$ for a given generic spherical realisation.
	Somewhat surprisingly, these two realisation numbers are not always equal.
	Recently developed algorithms for computing realisation numbers determined that the inequality $c_2(G) \leq c_2^*(G)$ holds for any minimally 2-rigid graph $G$ with 12 vertices or less.
	In this paper we confirm that, for any dimension $d$, the inequality $c_d(G) \leq c_d^*(G)$ holds for every $d$-rigid graph $G$.
	This result is obtained via new techniques involving coning,
	the graph operation that adds an extra vertex adjacent to all original vertices of the graph.
\end{abstract}

\section{Introduction}\label{sec:intro}

A (finite simple) graph is said to be \emph{$d$-rigid} if every generic realisation of the graph in $d$-dimensional Euclidean space is \emph{rigid},
i.e., shares edge-lengths with at most finitely many other realisations in the same space modulo isometries.
A $d$-rigid graph is \emph{minimally $d$-rigid} if removing any edge of the graph forms a graph that is not $d$-rigid.
Given a $d$-rigid graph, we would wish to know how many possible edge length equivalent realisations exist for any given generic realisation solely from the structural properties of the graph.
This is unfortunately not possible for most graphs as the number of equivalent realisations differs between different generic realisations;
for example, see \Cref{fig:real}.
As with many problems in algebraic geometry, the solution is to extend the problem to allow complex solutions.
By doing so,
we can concretely define the number of equivalent complex realisations (modulo congruence) of a generic $d$-dimensional complex realisation of a graph $G$.
We refer to this number as the \emph{$d$-realisation number of $G$} and denote it by $c_d(G)$.
(See \Cref{def:count} for a rigorous definition of the concept.)

Here we must make an important technical point:
\textbf{we use the definition of a graph's $d$-realisation number given by Jackson and Owen \cite{JacksonOwen}, and as such we consider reflections of a realisation to be congruent also};
for example, the complete graph with $d+1$ vertices has a $d$-realisation number of 1.
The variant of $d$-realisation number used by Borcea and Streinu \cite{BorceaStreinu} and Capco et al.~\cite{PlaneCount} does for algebraic reasons count reflections, and so is exactly double the $d$-realisation number used here.
We have opted for the former definition of a $d$-realisation number since it preserves an important property of congruence:
any two equal-size ordered sets with the same pairwise distances between points are congruent.
We urge any reader who is cross-referencing with multiple sources to be careful about this technical point,
especially since many algorithms for computing $d$-realisation numbers use the latter definition (e.g., \cite{PlaneCount,SphereCount}).

\begin{figure}[ht]
 \centering
 \begin{tikzpicture}
    \begin{scope}
        \node[vertex] (1) at (0,0) {};
        \node[vertex] (2) at (0,1) {};
        \node[vertex] (3) at (-0.75,1.5) {};
        \node[vertex] (4) at (1.25,1.5) {};
        \path[name path=circ1] (3) circle[radius=1.5cm];
        \path[name path=circ2] (4) circle[radius=1.25cm];
        \path[name intersections={of=circ1 and circ2, by=x}];
        \node[vertex] (5) at (x) {};
        \draw[edge] (1)edge(2) (1)edge(3) (1)edge(4) (2)edge(3) (2)edge(4) (3)edge(5) (4)edge(5);
    \end{scope}
    \begin{scope}[xshift=3cm]
        \node[vertex] (1) at (0,0) {};
        \node[vertex] (2) at (0,1) {};
        \node[vertex] (3) at (-0.75,1.5) {};
        \node[vertex] (4) at (-1.25,1.5) {};
        \path[name path=circ1] (3) circle[radius=1.5cm];
        \path[name path=circ2] (4) circle[radius=1.25cm];
        \path[name intersections={of=circ1 and circ2, by=x}];
        \node[vertex] (5) at (x) {};
        \draw[edge] (1)edge(2) (1)edge(3) (1)edge(4) (2)edge(3) (2)edge(4) (3)edge(5) (4)edge(5);
    \end{scope}
    \begin{scope}[xshift=4cm]
     \draw[dashed,white!75!black] (0,-0.25)--(0,3);
    \end{scope}
    \begin{scope}[xshift=6cm]
        \node[vertex] (1) at (0,0) {};
        \node[vertex] (2) at (0,1) {};
        \node[vertex] (3) at (-0.75,1.5) {};
        \node[vertex] (4) at (-1.25,1.5) {};
        \path[name path=circ1] (3) circle[radius=0.75cm];
        \path[name path=circ2] (4) circle[radius=1cm];
        \path[name intersections={of=circ1 and circ2, by=x}];
        \node[vertex] (5) at (x) {};
        \draw[edge] (1)edge(2) (1)edge(3) (1)edge(4) (2)edge(3) (2)edge(4) (3)edge(5) (4)edge(5);
    \end{scope}
    \begin{scope}[xshift=8cm]
        \node[vertex] (1) at (0,0) {};
        \node[vertex] (2) at (0,1) {};
        \node[vertex] (3) at (-0.75,1.5) {};
        \node[vertex] (4) at (1.25,1.5) {};
        \draw[white!50!black,dotted] (3) circle[radius=0.75cm];
        \draw[white!50!black,dotted] (4) circle[radius=1cm];
        \draw[edge] (1)edge(2) (1)edge(3) (1)edge(4) (2)edge(3) (2)edge(4);
        \draw[edge,dashed] (3)--++(0.75,0) node (x1) {} (4)--++(-1,0) node (x2) {};
        \node[colR,font=\tiny] at ($(x1)!0.5!(x2)$) {\faBolt};
    \end{scope}
 \end{tikzpicture}
 \caption{Two realisations of the same graph on 5 vertices with different numbers of equivalent real realisations. The realisation on the left has 4 non-congruent equivalent real realisations (two of which are shown), whilst the one on the right has only 2 due to the edge lengths.}
 \label{fig:real}
\end{figure}
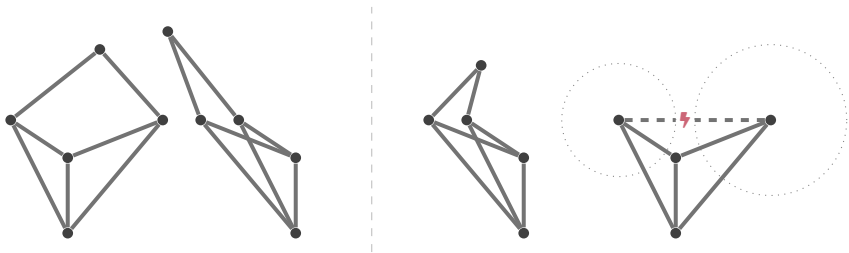

Via tools such as \emph{coning} --- adding a vertex to a graph adjacent to all other vertices, a process that preserves rigidity between dimensions --- Whiteley \cite{coning} proved that a graph is rigid when embedded generically on the $d$-dimensional sphere if and only if it is $d$-rigid.
Because of this equivalence, it is natural to ask how many equivalent spherical realisations exist for any given generic spherical realisation of a graph.
Yet again we are required to extend to complex solutions, where we also consider realisations on the complexification of the sphere.
We thus define the \emph{spherical $d$-realisation number} of a graph $G$, here denoted by $c^*_d(G)$, to be the number of equivalent complex spherical realisations (modulo congruence) of a generic $d$-dimensional complex spherical realisation of a graph $G$.
(See \Cref{def:countsphere} for a rigorous definition of the concept.)
As noted in \Cref{rem:hyperbolic},
this number is equal to the analogous realisation number for hyperbolic space,
and so can also be considered to be the non-Euclidean $d$-realisation number of a graph.

In recent years, deterministic algorithms have also been constructed for computing $c_2(G)$ \cite{PlaneCount} and $c_2^*(G)$ \cite{SphereCount} when the graph $G$ is a minimally 2-rigid.
It is also relatively easy to compute the values $c_1(G)$ and $c_1^*(G)$ for any graph $G$ (see \Cref{p:1d,p:1dsame}).
This is, unfortunately, where the good news stops:
when $d >2$, there exist no current feasible deterministic algorithms for computing either $c_d$ or $c_d^*$ for general graphs.
One probabilistic method for computing $c_d(G)$ and $c_d^*(G)$ when $G$ is $d$-rigid involves choosing a random realisation for the graph and applying Gr\"{o}bner basis computational techniques to the resulting algebraic solution set.
Whilst this algorithm can be used reliably to obtain a lower bound on the $d$-realisation number of a graph,
it is not deterministic and usually extremely slow (see \cite[Section 5]{PlaneCount} for computation speed comparisons).
Upper bounds can also be determined using mixed volume techniques \cite{Steffens2010} or multihomogenous B\'{e}zout bounds \cite{Bartzos2020}.

\subsection{Our contributions}

When computing $c_2$ and $c^*_2$ for various minimally 2-rigid graphs,
it was observed that the two numbers occasionally will differ;
see \Cref{fig:threeprismplane,fig:threeprismsphere} for the smallest 2-rigid graph where $c_2$ is strictly less than $c_2^*$.
We also observed that the opposite never held for any of the graphs whose realisation numbers were computed:
to be exact, for every minimally 2-rigid graph with at most 12 vertices, it was computed that the $2$-realisation number $c_2$ is never greater than the spherical $2$-realisation number $c_2^*(G)$:
this observation can be achieved by combining implementations of the algorithms \cite{ZenodoAlg,SphereAlg} and the data set of all minimally 2-rigid graphs with at most 12 vertices that can be found at \cite{ZenodoData}.
Similarly, randomized experiments with higher numbers of vertices also exhibited the exact same behaviour.
Our main contribution of this paper is proving that this observation is indeed true for any graph in any dimension.

  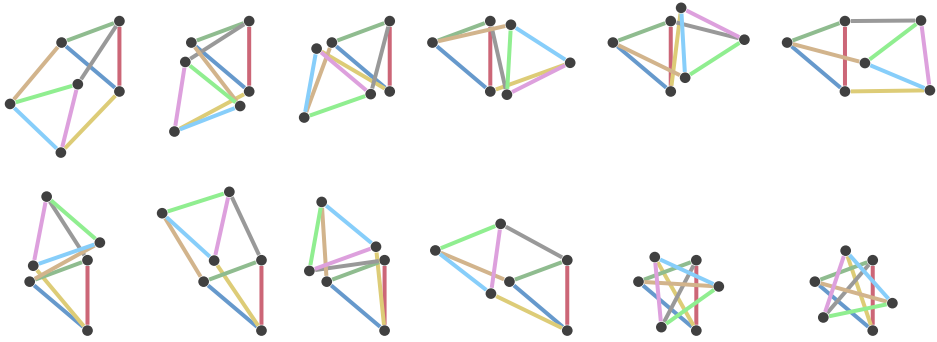
\begin{figure}[t]
    \centering
    \begin{tabular}{cccccc}
      \begin{tikzpicture}[scale=1]
    \draw[colbg] ($(0, -0.8098)-(0,0.3)$)--($(0, 1.1092)+(0,0.3)$);
	\node[vertex] (1) at (0., 0.) {};
	\node[vertex] (2) at (0., 0.93333) {};
	\node[vertex] (3) at (-0.76102, 0.64873) {};
	\node[vertex] (4) at (-0.77023, -0.80987) {};
	\node[vertex] (5) at (-0.54744, 0.09649) {};
	\node[vertex] (6) at (-1.43841, -0.16506) {};
	\draw[edge,colr1] (1)edge(2);
	\draw[edge,colr2] (1)edge(3);
	\draw[edge,colr3] (1)edge(4);
	\draw[edge,colr4] (2)edge(3);
	\draw[edge,colr5] (2)edge(5);
	\draw[edge,colr6] (3)edge(6);
	\draw[edge,colr7] (4)edge(5);
	\draw[edge,colr8] (4)edge(6);
	\draw[edge,colr9] (5)edge(6);
\end{tikzpicture}
&
\begin{tikzpicture}[scale=1]
	\draw[colbg] ($(0, -0.8098)-(0,0.3)$)--($(0, 1.1092)+(0,0.3)$);
	\node[vertex] (1) at (0., 0.) {};
	\node[vertex] (2) at (0., 0.93333) {};
	\node[vertex] (3) at (-0.76102, 0.64873) {};
	\node[vertex] (4) at (-0.98344, -0.53101) {};
	\node[vertex] (5) at (-0.84035, 0.39129) {};
	\node[vertex] (6) at (-0.1186, -0.19294) {};
	\draw[edge,colr1] (1)edge(2);
	\draw[edge,colr2] (1)edge(3);
	\draw[edge,colr3] (1)edge(4);
	\draw[edge,colr4] (2)edge(3);
	\draw[edge,colr5] (2)edge(5);
	\draw[edge,colr6] (3)edge(6);
	\draw[edge,colr7] (4)edge(5);
	\draw[edge,colr8] (4)edge(6);
	\draw[edge,colr9] (5)edge(6);
\end{tikzpicture}
&
\begin{tikzpicture}[scale=1]
	\draw[colbg] ($(0, -0.8098)-(0,0.3)$)--($(0, 1.1092)+(0,0.3)$);
	\node[vertex] (1) at (0., 0.) {};
	\node[vertex] (2) at (0., 0.93333) {};
	\node[vertex] (3) at (-0.76102, 0.64873) {};
	\node[vertex] (4) at (-0.96216, 0.56867) {};
	\node[vertex] (5) at (-0.25021, -0.03486) {};
	\node[vertex] (6) at (-1.12529, -0.34546) {};
	\draw[edge,colr1] (1)edge(2);
	\draw[edge,colr2] (1)edge(3);
	\draw[edge,colr3] (1)edge(4);
	\draw[edge,colr4] (2)edge(3);
	\draw[edge,colr5] (2)edge(5);
	\draw[edge,colr6] (3)edge(6);
	\draw[edge,colr7] (4)edge(5);
	\draw[edge,colr8] (4)edge(6);
	\draw[edge,colr9] (5)edge(6);
\end{tikzpicture}
&
\begin{tikzpicture}[scale=1]
	\draw[colbg] ($(0, -0.8098)-(0,0.3)$)--($(0, 1.1092)+(0,0.3)$);
	\node[vertex] (1) at (0., 0.) {};
	\node[vertex] (2) at (0., 0.93333) {};
	\node[vertex] (3) at (-0.76102, 0.64873) {};
	\node[vertex] (4) at (1.05091, 0.38043) {};
	\node[vertex] (5) at (0.21886, -0.04242) {};
	\node[vertex] (6) at (0.27118, 0.88467) {};
	\draw[edge,colr1] (1)edge(2);
	\draw[edge,colr2] (1)edge(3);
	\draw[edge,colr3] (1)edge(4);
	\draw[edge,colr4] (2)edge(3);
	\draw[edge,colr5] (2)edge(5);
	\draw[edge,colr6] (3)edge(6);
	\draw[edge,colr7] (4)edge(5);
	\draw[edge,colr8] (4)edge(6);
	\draw[edge,colr9] (5)edge(6);
\end{tikzpicture}
&
\begin{tikzpicture}[scale=1]
	\draw[colbg] ($(0, -0.8098)-(0,0.3)$)--($(0, 1.1092)+(0,0.3)$);
	\node[vertex] (1) at (0., 0.) {};
	\node[vertex] (2) at (0., 0.93333) {};
	\node[vertex] (3) at (-0.76102, 0.64873) {};
	\node[vertex] (4) at (0.13693, 1.10923) {};
	\node[vertex] (5) at (0.96907, 0.68656) {};
	\node[vertex] (6) at (0.18946, 0.18214) {};
	\draw[edge,colr1] (1)edge(2);
	\draw[edge,colr2] (1)edge(3);
	\draw[edge,colr3] (1)edge(4);
	\draw[edge,colr4] (2)edge(3);
	\draw[edge,colr5] (2)edge(5);
	\draw[edge,colr6] (3)edge(6);
	\draw[edge,colr7] (4)edge(5);
	\draw[edge,colr8] (4)edge(6);
	\draw[edge,colr9] (5)edge(6);
\end{tikzpicture}
&
\begin{tikzpicture}[scale=1]
	\draw[colbg] ($(0, -0.8098)-(0,0.3)$)--($(0, 1.1092)+(0,0.3)$);
	\node[vertex] (1) at (0., 0.) {};
	\node[vertex] (2) at (0., 0.93333) {};
	\node[vertex] (3) at (-0.76102, 0.64873) {};
	\node[vertex] (4) at (1.11754, 0.01527) {};
	\node[vertex] (5) at (0.99997, 0.94116) {};
	\node[vertex] (6) at (0.26236, 0.37709) {};
	\draw[edge,colr1] (1)edge(2);
	\draw[edge,colr2] (1)edge(3);
	\draw[edge,colr3] (1)edge(4);
	\draw[edge,colr4] (2)edge(3);
	\draw[edge,colr5] (2)edge(5);
	\draw[edge,colr6] (3)edge(6);
	\draw[edge,colr7] (4)edge(5);
	\draw[edge,colr8] (4)edge(6);
	\draw[edge,colr9] (5)edge(6);
\end{tikzpicture}
\\
\begin{tikzpicture}[scale=1]
	\node[vertex] (1) at (0., 0.) {};
	\node[vertex] (2) at (0., 0.93333) {};
	\node[vertex] (3) at (-0.76102, 0.64873) {};
	\node[vertex] (4) at (-0.71395, 0.85989) {};
	\node[vertex] (5) at (-0.53774, 1.77644) {};
	\node[vertex] (6) at (0.16251, 1.16661) {};
	\draw[edge,colr1] (1)edge(2);
	\draw[edge,colr2] (1)edge(3);
	\draw[edge,colr3] (1)edge(4);
	\draw[edge,colr4] (2)edge(3);
	\draw[edge,colr5] (2)edge(5);
	\draw[edge,colr6] (3)edge(6);
	\draw[edge,colr7] (4)edge(5);
	\draw[edge,colr8] (4)edge(6);
	\draw[edge,colr9] (5)edge(6);
\end{tikzpicture}
&
\begin{tikzpicture}[scale=1]
	\node[vertex] (1) at (0., 0.) {};
	\node[vertex] (2) at (0., 0.93333) {};
	\node[vertex] (3) at (-0.76102, 0.64873) {};
	\node[vertex] (4) at (-0.62275, 0.92807) {};
	\node[vertex] (5) at (-0.42254, 1.83968) {};
	\node[vertex] (6) at (-1.30675, 1.55608) {};
	\draw[edge,colr1] (1)edge(2);
	\draw[edge,colr2] (1)edge(3);
	\draw[edge,colr3] (1)edge(4);
	\draw[edge,colr4] (2)edge(3);
	\draw[edge,colr5] (2)edge(5);
	\draw[edge,colr6] (3)edge(6);
	\draw[edge,colr7] (4)edge(5);
	\draw[edge,colr8] (4)edge(6);
	\draw[edge,colr9] (5)edge(6);
\end{tikzpicture}
&
\begin{tikzpicture}[scale=1]
	\node[vertex] (1) at (0., 0.) {};
	\node[vertex] (2) at (0., 0.93333) {};
	\node[vertex] (3) at (-0.76102, 0.64873) {};
	\node[vertex] (4) at (-0.11333, 1.11189) {};
	\node[vertex] (5) at (-0.98985, 0.79121) {};
	\node[vertex] (6) at (-0.82741, 1.70547) {};
	\draw[edge,colr1] (1)edge(2);
	\draw[edge,colr2] (1)edge(3);
	\draw[edge,colr3] (1)edge(4);
	\draw[edge,colr4] (2)edge(3);
	\draw[edge,colr5] (2)edge(5);
	\draw[edge,colr6] (3)edge(6);
	\draw[edge,colr7] (4)edge(5);
	\draw[edge,colr8] (4)edge(6);
	\draw[edge,colr9] (5)edge(6);
\end{tikzpicture}
&
\begin{tikzpicture}[scale=1]
	\node[vertex] (1) at (0., 0.) {};
	\node[vertex] (2) at (0., 0.93333) {};
	\node[vertex] (3) at (-0.76102, 0.64873) {};
	\node[vertex] (4) at (-1.00434, 0.49035) {};
	\node[vertex] (5) at (-0.87642, 1.41487) {};
	\node[vertex] (6) at (-1.73559, 1.06263) {};
	\draw[edge,colr1] (1)edge(2);
	\draw[edge,colr2] (1)edge(3);
	\draw[edge,colr3] (1)edge(4);
	\draw[edge,colr4] (2)edge(3);
	\draw[edge,colr5] (2)edge(5);
	\draw[edge,colr6] (3)edge(6);
	\draw[edge,colr7] (4)edge(5);
	\draw[edge,colr8] (4)edge(6);
	\draw[edge,colr9] (5)edge(6);
\end{tikzpicture}
&
\begin{tikzpicture}[scale=1]
	\node[vertex] (1) at (0., 0.) {};
	\node[vertex] (2) at (0., 0.93333) {};
	\node[vertex] (3) at (-0.76102, 0.64873) {};
	\node[vertex] (4) at (-0.54683, 0.97474) {};
	\node[vertex] (5) at (-0.46006, 0.04545) {};
	\node[vertex] (6) at (0.29587, 0.58472) {};
	\draw[edge,colr1] (1)edge(2);
	\draw[edge,colr2] (1)edge(3);
	\draw[edge,colr3] (1)edge(4);
	\draw[edge,colr4] (2)edge(3);
	\draw[edge,colr5] (2)edge(5);
	\draw[edge,colr6] (3)edge(6);
	\draw[edge,colr7] (4)edge(5);
	\draw[edge,colr8] (4)edge(6);
	\draw[edge,colr9] (5)edge(6);
\end{tikzpicture}
&
\begin{tikzpicture}[scale=1]
	\node[vertex] (1) at (0., 0.) {};
	\node[vertex] (2) at (0., 0.93333) {};
	\node[vertex] (3) at (-0.76102, 0.64873) {};
	\node[vertex] (4) at (-0.35606, 1.05941) {};
	\node[vertex] (5) at (-0.65033, 0.17369) {};
	\node[vertex] (6) at (0.25864, 0.36343) {};
	\draw[edge,colr1] (1)edge(2);
	\draw[edge,colr2] (1)edge(3);
	\draw[edge,colr3] (1)edge(4);
	\draw[edge,colr4] (2)edge(3);
	\draw[edge,colr5] (2)edge(5);
	\draw[edge,colr6] (3)edge(6);
	\draw[edge,colr7] (4)edge(5);
	\draw[edge,colr8] (4)edge(6);
	\draw[edge,colr9] (5)edge(6);
\end{tikzpicture}
    \end{tabular}
    \caption{12 equivalent real realisations of the 3-prism graph.}
    \label{fig:threeprismplane}
  \end{figure}
  \begin{figure}[t]
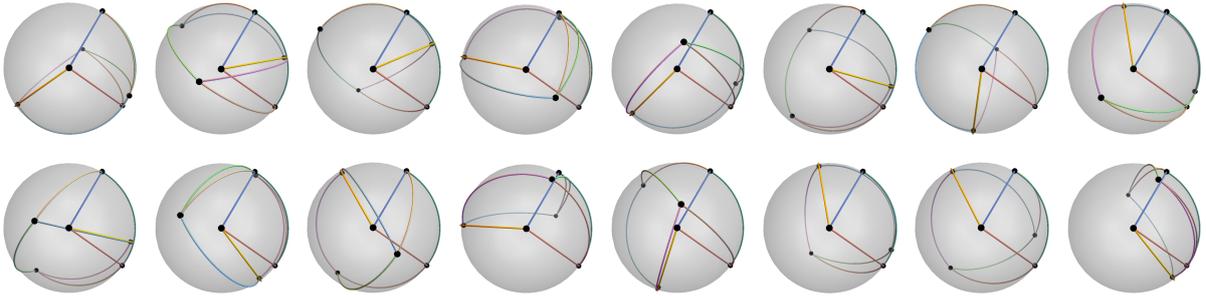

    \centering
      \foreach \i in {1,2,...,8} {\includegraphics[width=1.75cm]{RealSphereRealization\i.png}\hspace{0.25cm}}\\[2ex]
      \foreach \i in {9,10,...,16} {\includegraphics[width=1.75cm]{RealSphereRealization\i.png}\hspace{0.25cm}}\vspace{-0.25cm}
    \caption{16 equivalent real realisations of the 3-prism graph on the sphere.}
    \label{fig:threeprismsphere}
  \end{figure}

\begin{theorem}\label{t:sphereplane}
	For any graph $G$ and any positive integer $d$,
	we have $c_{d}(G) \leq c_d^*(G)$.
\end{theorem}

To prove this result, we first show an equivalence between the spherical $d$-realisation number of a graph $G$ and the $(d+1)$-realisation number of the coned graph $G*o$, the graph formed by adding a new vertex adjacent to all other vertices of $G$.

\begin{theorem}\label{t:conecount}
	Let $d$ be a positive integer and let $G*o$ be a coning of a graph $G$.
	Then $c^*_{d}(G) = c_{d+1}(G*o)$.
\end{theorem}

Our proof for \Cref{t:conecount} hinges on the observation that realisation numbers for frameworks are, in the special case of coned frameworks, projectively invariant.
\Cref{t:sphereplane}, however, requires a more specialised approach to graph coning.
We actually prove that $c_{d}(G) \leq c_{d+1}(G*o)$,
since \Cref{t:conecount} then implies the required inequality.
Our technique now is as follows.
We construct a very specific space of $(d+1)$-dimensional realisations of the coned graph $G*o$ which include realisations where the cone vertex $o$ is mapped to a point at infinity.
For this special class of realisations, the distance constraints stemming from edges connecting the cone vertex $o$ to each vertex $v$ of $G$ are instead linear constraints that force the vertices of $G$ to lie in a family of parallel hyperplanes.
This allows us to embed our $d$-dimensional realisations of $G$ into our new larger space of $(d+1)$-dimensional realisations of $G*o$.
From this, we can then approximate any framework $(G,p)$ by a sequence of frameworks $(G*o,p')$ in our larger space such that the required geometric information regarding realisation spaces (i.e., an upper bound on the number of points) is preserved.

The idea of considering a linear constraint to be a point at infinity is not new;
see for example \cite{EJNSTW19}.
Our contribution is to construct a well-defined set of tools that allows us to approximate $d$-dimensional general realisations of a graph $G$ by $(d+1)$-dimensional coned realisations of the cone graph $G*o$ in a way that preserves some geometric information about their realisation spaces:
in our particular case, the geometric information that is preserved is an upper bound on the size of the set.
Since prior techniques with coning have been focused on the preservation of linear information of realisations (for example, the space of self-stresses of the realisation),
such results have not been previously possible.
We believe our new enhanced coning technique has the potential to applied to other generic invariants related to framework rigidity (see \Cref{rem:alginfo} for further discussion).

\subsection{Structure of paper}

Our paper is structured as follows.
In \Cref{sec:prelim} we cover all the necessary background surrounding the topics of rigidity, spherical rigidity and coning.
In \Cref{sec:count,sec:countsphere} we provide rigorous definitions for the $d$-realisation number and spherical $d$-realisation number respectively.
Whilst such ideas have been alluded to in prior papers (for example \cite{PlaneCount,JacksonOwen}),
they have previously restricted themselves to 2-dimensional spaces only.
Due to the various technicalities that occur when moving between our various higher-dimensional geometries,
we have opted to include thorough proofs of all necessary background results.
In \Cref{sec:cone} and \Cref{sec:flat} we prove \Cref{t:conecount} and \Cref{t:sphereplane} respectively.
In \Cref{sec:cone} we also prove that repeatedly coning a graph will in some sense stabilise its realisation numbers (\Cref{t:conecount2}),
which can be utilised to construct infinite families of graphs for each dimension with realisation numbers that can be computed deterministically.
We conclude the paper by discussing various computational results we have obtained in \Cref{sec:compute}.





\section{Preliminary results for rigidity theory}\label{sec:prelim}

In this section we cover the necessary background results for discussing rigidity in geometries based over the real numbers.
Many of these concepts are extended to geometries based over the complex numbers in \Cref{sec:count,sec:countsphere}.

\subsection{Euclidean space rigidity}

A \emph{realisation} of a (finite simple) graph $G=(V,E)$ in $\mathbb{R}^d$ is a map $p:V \rightarrow \mathbb{R}^d$,
and the linear space of all realisations of a graph is denoted by $(\mathbb{R}^d)^V$.
A realisation $p$ is said to be \emph{generic} if the set of coordinates of $p$ forms an algebraic independent set of $d|V|$ elements.\footnote{The requirement that the algebraic independent set has $d|V|$ elements is stated to avoid the possibility of identical coordinates.}
A graph-realisation pair $(G,p)$ is said to be a \emph{$d$-dimensional framework}.
Two frameworks $(G,p)$ and $(G,q)$ in $\mathbb{R}^d$ are said to be \emph{equivalent} if (given $\|\cdot\|$ is the standard Euclidean norm) $\|p_v - p_w\| = \|q_v-q_w\|$ holds for every edge $vw \in E$.
A framework $(G,p)$ in $\mathbb{R}^d$ is now said to be \emph{rigid} if the following holds for some $\varepsilon >0$:
if $(G,q)$ is an equivalent framework in $\mathbb{R}^d$ such that $\|p_v-q_w\|<\varepsilon$ for all $v \in V$,
then there exists an isometry $f: \mathbb{R}^d \rightarrow \mathbb{R}^d$ such that $q_v = f(p_v)$ for all $v \in V$.

Determining whether a framework is rigid is NP-Hard for $d \geq 2$ \cite{Abbot}.
To combat this, we construct the \emph{rigidity matrix} $R(G,p)$ of a given framework $(G,p)$.
This is the $|E| \times d|V|$ matrix with the row labelled $vw \in E$ given by
\begin{align*}
	[ 0 \quad \cdots  \quad 0  \quad \overbrace{p_v-p_w}^{v} \quad 0  \quad \cdots  \quad 0  \quad \overbrace{p_w-p_v}^{w} \quad 0  \quad \cdots  \quad 0].
\end{align*}
A framework $(G,p)$ is said to be \emph{regular} if its rigidity matrix has maximal rank,
i.e., for any $d$-dimensional realisation $q$ of $G$ we have $\rank R(G,p) \geq \rank R(G,q)$.
Any framework $(G,p)$ with a surjective rigidity matrix (i.e., $\rank (G,p)=|E|$) is automatically regular:
in such a case we say that the framework $(G,p)$ is \emph{independent}.
Any generic framework must also be regular, since the non-regular frameworks form an algebraic set defined by rational-coefficient polynomials.

\begin{theorem}[\cite{AsimowRothI}]\label{t:asimowroth}
	Let $(G,p)$ be a regular framework in $\mathbb{R}^d$.
	Then the following properties are equivalent.
	\begin{enumerate}
		\item \label{t:asimowroth1} $(G,p)$ is rigid.
		\item \label{t:asimowroth2} Either $G$ has at least $d$ vertices and $\rank R(G,p) = d|V| - \binom{d+1}{2}$, or $G$ is a complete graph.
	\end{enumerate}
\end{theorem}

Since every generic framework is regular,
\Cref{t:asimowroth} informs us that both rigidity and independence are generic properties.
This motivates the following definitions.

\begin{definition}
	A graph $G$ is said to be \emph{$d$-rigid} (respectively, $d$-independent) if there exists a $d$-dimensional rigid (respectively, independent) generic framework $(G,p)$.
	If $G$ is rigid but $G-e$ is not for each edge $e \in E$,
	then $G$ is said to be \emph{minimally $d$-rigid}.
\end{definition}
%
%

One method that can be used to construct rigid/independent graphs in higher dimensions is via coning.
Specifically:
given a graph $G=(V,E)$,
we define the \emph{cone of $G$ (by the new vertex $o$)} to be the graph $G*o := (V*o, E*o)$ where $o \notin V$, $V*o := V \cup \{o\}$ and 
\begin{align*}
	E*o := E \cup \{ o v : v \in V\}.
\end{align*}
We can now use the next result to move between dimensions whilst preserving rigidity and independence.

\begin{theorem}[{\cite[Theorem 5]{coning}}]\label{t:cone}
	A graph $G=(V,E)$ is $d$-rigid (respectively, $d$-independent) if and only if $G*o$ is $(d+1)$-rigid (respectively, $(d+1)$-independent).
\end{theorem}

\subsection{Spherical space rigidity}

The \emph{$d$-dimensional (unit) sphere} is the set
\begin{align*}
	\mathbb{S}^d := \left\{ x \in \mathbb{R}^{d+1} : \|x\|^2=1 \right\}.
\end{align*}
Similar to the Euclidean case, we define a \emph{$d$-dimensional spherical realisation} to be a pair $(G,p)$ of a graph $G=(V,E)$ and a \emph{$d$-dimensional spherical realisation} $p: V \rightarrow \mathbb{S}^d$.
The notions of equivalence and rigidity can be analogously defined for spherical frameworks by restricting all realisations to be spherical realisations.

An important observation is that any $d$-dimensional spherical realisation of a graph $G$ can (for rigidity purposes at least) be considered to be a $(d+1)$-dimensional realisation of the cone of $G$.
To be more specific:
given a $d$-dimensional spherical framework $(G,p)$,
we define the $(d+1)$-dimensional framework $(G*o,p')$ by setting $p'_v=(p_v,1)$ for all $v \in V$ and $p'_o = \mathbf{0}$.
From this construction it is easy to see that the framework $(G*o,p')$ is rigid if and only if $(G,p)$ is rigid.
Slightly less obviously,
this remains true if we scale each $p_v$ by some positive scalar.

\begin{proposition}\label{p:slide}
	For a graph $G=(V,E)$,
	let $p,q$ be $d$-dimensional realisations of the cone $G*o$ where $p_o=q_o = \mathbf{0}$.
	Further suppose that for each $v \in V$,
	there exists a scalar $r_v >0$ such that $q_v = r_v p_v$.
	Then $(G*o,p)$ is rigid if and only if $(G*o,q)$ is rigid,
	and $\rank R(G*o,p) = \rank R(G*o,q)$.
\end{proposition}

\begin{proof}
	That $\rank R(G*o,p) = \rank R(G*o,q)$ follows from the observation that the rank of a rigidity matrix is projectively invariant;
	see \cite[Theorem 1]{Izm09} for more details.
	Now choose any $d$-dimensional framework $(G*o,\tilde{p})$ equivalent to $(G*o,p)$ and define $\tilde{q}$ to be the realisation of $G*o$ with $\tilde{q}_v = r_v \tilde{p}_v$ for each $v \in V$.
	For every edge $vw$ of $G$ we have
	\begin{gather*}
		\|\tilde{p}_v\|^2 = \|\tilde{p}_v - \tilde{p}_o\|^2 = \|p_v - p_o\|^2 = \|p_v\|^2, \qquad
		\|\tilde{p}_w\|^2 = \|\tilde{p}_w - \tilde{p}_o\|^2 = \|p_w - p_o\|^2 = \|p_w\|^2,\\		
		\tilde{p}_v \cdot \tilde{p}_v =  \|\tilde{p}_v\|^2 + \|\tilde{p}_w\|^2 - \| \tilde{p}_v - \tilde{p}_w\|^2= \|p_v\|^2 + \|p_w\|^2 - \| p_v - p_w\|^2 = p_v \cdot p_w.
	\end{gather*}
	(Remember that $ov , ow \in E*o$ and $\tilde{p}_o = p_o = \mathbf{0}$.)
	Using the above equalities,
	we see that
	\begin{align*}
		\|\tilde{q}_v-\tilde{q}_w\|^2 &= \|r_v\tilde{p}_v-r_w\tilde{p}_w\|^2 \\
		&= \|r_v\tilde{p}_v\|^2 + \|r_w\tilde{p}_w\|^2 - r_v r_w\tilde{p}_v \cdot \tilde{p}_w \\
		&= \|r_vp_v\|^2 + \|r_wp_w\|^2 - r_v r_wp_v \cdot p_w  \\
		&= \|r_vp_v-r_wp_w\|^2 \\
		&= \|q_v-q_w\|^2,
	\end{align*}
	and hence $(G*o,q)$ and $(G*o,\tilde{q})$ are equivalent.
	From this it follows that $(G*o,p)$ is rigid if and only if $(G*o,q)$ is rigid.
\end{proof}

From a combination of \Cref{t:asimowroth}, \Cref{t:cone} and \Cref{p:slide},
we can now see that a graph is rigid in $d$-dimensional Euclidean space if and only if it is rigid in $d$-dimensional spherical space.

\begin{theorem}\label{t:sphereclassic}
	For any graph $G$, the following properties are equivalent.
	\begin{enumerate}
		\item $G$ is $d$-rigid.
		\item Almost every (i.e., with Lebesgue measure zero complement) $d$-dimensional spherical framework $(G,p)$ is rigid.
	\end{enumerate}
\end{theorem}

\begin{remark}
	It was first observed by Pogorelov \cite[Chapter V]{Pog73} that the linear space of infinitesimal motions of a $d$-dimensional spherical framework is isomorphic to the linear space of infinitesimal motions of the $d$-dimensional framework obtained by a central projection to Euclidean space.
	An alternative proof of \Cref{t:sphereclassic} now stems from the spherical analogue to \Cref{t:asimowroth},
	which can be proven in an almost identical way.
\end{remark}

\section{Counting complex realisations}\label{sec:count}

Our aim in this section is to prove that the definition of the $d$-realisation number for a graph alluded to in the introduction can be stated in a rigorous manner (see \Cref{def:count}).

\subsection{Complex rigidity map}

We recall that an algebraic set is a subset $A \subset \mathbb{C}^n$ of common zeroes of an ideal $I \subset \mathbb{C}[X_1,\ldots,X_n]$,
and a variety is an irreducible algebraic set.
If an algebraic set forms a smooth submanifold of $\mathbb{C}^n$ then it is said to be smooth.
We shall refer to a subset $U$ of an algebraic set $A$ as a Zariski closed/open/dense subset if $U$ is a closed/open/dense subset of $A$ with respect to the Zariski topology.

For every point  $x=(x_1,\ldots,x_d) \in \mathbb{C}^d$,
we define $[x]_i := x_i$ for each $i =1,\ldots,d$.
We now also consider realisations in $\mathbb{C}^d$ by extending $(\mathbb{R}^d)^V$ to the set $(\mathbb{C}^d)^V$.
Any realisation in $(\mathbb{R}^d)^V$ is said to be a \emph{real realisation}.
We extend the square of the Euclidean norm to $\mathbb{C}^d$ by defining $\|x\|^2 := \sum_{i=1}^d [x]_i^2$.
We note that this is a quadratic form and not the square of the standard complex norm,
and so can take complex values.
Later we shall also require the bilinear map $x \cdot y := \sum_{i =1}^d [x]_i [y]_i$,
where we notice that $x \cdot x = \| x\|^2$.

For any graph $G=(V,E)$ we define the \emph{complex rigidity map} to be the multivariable map
\begin{align*}
	f_{G,d} : (\mathbb{C}^{d})^V \rightarrow \mathbb{C}^E, ~ p \mapsto \left( \frac{1}{2}\|p_v-p_w\|^2 \right)_{vw \in E}.
\end{align*}
We denote the Zariski closure of the image of $f_{G,d}$ by $\ell_d(G)$.
Since the domain of $f_{G,d}$ is irreducible,
$\ell_d(G)$ is a variety.
Note that two realisations $p,q$ of $G$ in $\mathbb{R}^d$ are equivalent if and only if $f_{G,d}(p)=f_{G,d}(q)$.
Given $O(d, \mathbb{C})$ is the group of $d \times d$ complex-valued matrices $M$ where $M^T M = M M^T = I$,
we define two realisations $p,q \in (\mathbb{C}^{d})^V$ to be \emph{congruent} (denoted by $p \sim q$) if and only if there exists an $A \in O(d, \mathbb{C})$ and $x \in \mathbb{C}^{d}$ so that $p_v = Aq_v + x$ for all $v \in V$.
If the set of vertices of $(G,p)$ affinely span $\mathbb{C}^d$,
we have the following equivalent statement:
two realisations $p,q$ are congruent if and only if $f_{K_V,d}(p) = f_{K_V,d}(q)$ (see \cite[Section 10]{Gortler2014} for more details).
For all $p \in (\mathbb{C}^{d})^V$,
we define 
\begin{align*}
	C_d(G,p) := f^{-1}_{G,d} (f_{G,d}(p))/_\sim
\end{align*}
to be the \emph{realisation space of $(G,p)$}.

Our new definitions might look a little strange at first.
For example,
the elements of $O(d, \mathbb{C})$ are not isometries of $\mathbb{C}^d$ (i.e., they are not unitary matrices),
although $\|Ax\|^2 = \|x\|^2$ for all $x \in \mathbb{C}^d$.
However, our previous definitions of independence and rigidity can be encoded in our new language of morphisms between complex spaces.
We first recall that a morphism $f : X \rightarrow Y$ between algebraic sets $X \subset \mathbb{C}^m$ and $Y \subset \mathbb{C}^n$ (i.e.~the restriction of a polynomial map $\mathbb{C}^m \rightarrow \mathbb{C}^n$) is \emph{dominant} if $Y \setminus f(X)$ is contained in a Zariski closed proper subset of $Y$;
see \Cref{sec:dom} for more details regarding dominant morphisms.

\begin{lemma}\label{l:domind}
	Let $G=(V,E)$ be any graph.
	Then the following are equivalent:
	\begin{enumerate}
		\item \label{l:domind1} $G$ is $d$-independent.
		\item \label{l:domind2} The map $f_{G,d}$ is dominant.
		\item \label{l:domind3} $\ell_d(G) = \mathbb{C}^E$.
	\end{enumerate}
\end{lemma} 

\begin{proof}
	By the definition of a dominant map, \ref{l:domind2} and \ref{l:domind3} are equivalent.
	The map $f_{G,d}$ is dominant if and only if $\rank \diff f_{G,d}(p) = |E|$ for some $p \in (\mathbb{C}^{d})^V$ (see \Cref{borel91}).
	If $G$ is $d$-independent then there exists a $p \in (\mathbb{R}^{d})^V$ such that $\rank \diff f_{G,d}(p) = \rank R(G,p) = |E|$,
	hence \ref{l:domind1} implies \ref{l:domind2}.
	Suppose that $G$ is not $d$-independent;
	i.\,e., for each $p \in (\mathbb{R}^{d})^V$ we have $\rank \diff f_{G,d}(p) < |E|$.
	Then the set $X := \{ p \in (\mathbb{C}^{d})^V: \rank \diff f_{G,d}(p) < |E| \}$ is a Zariski closed subset of $(\mathbb{C}^{d})^V$ that contains $(\mathbb{R}^{d})^V$.
	Hence, $X = (\mathbb{C}^{d})^V$,
	and \ref{l:domind2} implies \ref{l:domind1} as required.
\end{proof}

Let $G=(V,E)$ be a graph with at least $d+1$ vertices and fix a sequence of $d$ vertices $v_1,\ldots, v_d$.
We now define the algebraic set
\begin{align}\label{eq:xset}
	X_{G,d} := \left\{ p \in (\mathbb{C}^{d})^V : [p_{v_k}]_j=0 \text{ if } j \geq k  \right\}\,.
\end{align}
(At certain points in the paper we use vertices $w_1,\ldots,w_d$ to define $X_{G,d}$,
which can be done by replacing each $v_i$ with $w_i$ in the above definition.)
We further note that $X_{G,d}$ has dimension $d|V| - \binom{d+1}{2}$.
Since $X_{G,d}$ is defined by a set of linear equations,
it is irreducible.
With this, we define the morphism
\begin{align*}
	\tilde{f}_{G,d} : X_{G,d} \rightarrow \ell_d(G), ~ p \mapsto f_{G,d}(p),
\end{align*}
i.\,e., the restriction of $f_{G,d}$ to the domain $X_{G,d}$ and the codomain $\ell_d(G)$.

\begin{lemma}\label{l:rotate}
	Let $G=(V,E)$ be a graph with $|V| \geq d+1$,
	and fix a sequence of $d$ vertices $v_1,\ldots, v_d$.
	For each realisation $p \in (\mathbb{C}^d)^V$,
	define the $(d-1) \times (d-1)$ symmetric matrix
	\begin{align*}
		\mathbb{G}(p) := 		
		\begin{bmatrix}
			(p_{v_2} - p_{v_1})^T \\
			\vdots \\
			(p_{v_d} - p_{v_1})^T
		\end{bmatrix}
		\begin{bmatrix}
			p_{v_2} - p_{v_1} & \cdots & p_{v_d} - p_{v_1}
		\end{bmatrix}.
	\end{align*}
	Then there exists a realisation $q \in X_{G,d}$ congruent to $p$ if $\mathbb{G}(p)$ only has non-zero leading principal minors.\footnote{A \emph{leading principal minor (of order $n$)} of a square matrix is the determinant of the matrix formed by taking the first $n$ rows and columns.}
\end{lemma}

The conditions stated in \Cref{l:rotate} seem rather bizarre,
especially since no such conditions are required if we restrict ourselves to real realisations.
It is also rather easy to see they are not necessary either:
simply choose any realisation $p \in X_{G,d}$ where $p_{v_i} = \mathbf{0}$ for each $i \in \{1,\ldots,d\}$.
To see why we need to be so cautious,
take $G=(V,E)$ to be the complete graph with 4 vertices $v_1,v_2,v_3,v_4$,
and let $p$ be the 3-dimensional realisation of $G$ where
\begin{align*}
	p_{v_1} = (0,0,0), \qquad p_{v_2} = (1,0,0), \qquad p_{v_3} = (2,1,i), \qquad p_{v_4} = (0,1,0).
\end{align*}
In this specific case, there are no realisations in $X_{G,d}$ that are congruent to $p$.
The framework $(G,p)$ also has the interesting properties that its vertices affinely span $\mathbb{C}^3$ and $\|p_v-p_w\|^2 >0$ for every edge $vw \in E$.
Since the matrix
\begin{align*}
	\mathbb{G}(p) =
	\begin{bmatrix}
		1 & 0 & 0 \\
		2 & 1 & i
	\end{bmatrix}
	\begin{bmatrix}
		1 & 2 \\
		0 & 1 \\
		0 & i
	\end{bmatrix}
	=
	\begin{bmatrix}
		1 & 2 \\
		2 & 4
	\end{bmatrix}
\end{align*}
has rank 1,
our constructed realisation is not a counter-example to \Cref{l:rotate}.

\begin{proof}
	Fix $p \in (\mathbb{C}^d)^V$ with the stipulated properties.
	With this,
	define $p^1 \in (\mathbb{C}^d)^V$ to be the realisation where $p^1_v = p_v - p_{v_1}$ for each $v \in V$.
	If $d=1$ then $p^1 \in X_{G,1}$,
	so suppose that $d \geq 2$.
	We now observe that the matrix $\mathbb{G}(p^1) = \mathbb{G}(p)$.
	In fact, a much stronger statement is true:
	if $q$ is a congruent realisation then $\mathbb{G}(q) = \mathbb{G}(p)$.
	
	We now form the sequence of congruent realisations $p^1,\ldots,p^d$ in the following inductive way.
	Fix $n \in \{2,\ldots,d\}$,
	and suppose that $p^1,\ldots,p^{n-1}$ have already been constructed such that $[p^{n-1}_{v_k}]_j =0$ if $k \leq \min \{j, n-1\}$.
	For each $k \in \{1,\ldots,d\}$,
	fix $e_k$ to be the vector with $[e_k]_j = 1$ if $j=k$ and $[e_k]_j = 0$ otherwise.
	We observe here that the linear space spanned by the vectors $e_1,\ldots, e_{n-2}$ contains the points $p^{n-1}_{v_2},\ldots, p^{n-1}_{v_{n-1}}$.
	Fix $z$ to be the vector formed from $p^{n-1}_{v_{n}}$ by replacing its first $n-2$ coordinates with zeroes.
	It is immediate that $z \cdot e_j = 0$ for all $j \in \{1 , \ldots, n-2\}$.
	
	Suppose for contradiction that $\|z\|^2 = 0$,
	and fix $y = p^{n-1}_{v_{n}} - z$.
	Set $A = [p^{n-1}_{v_2} ~ \cdots ~ p^{n-1}_{v_{n-1}}]$.
	The determinant of $A^T A$ is the leading principal minor of order $n-2$ of $\mathbb{G}(p) = \mathbb{G}(p^{n-1})$,
	and so is non-zero.
	Hence, the column span of $A$ is exactly the linear space spanned by $e_1,\ldots, e_{n-2}$;
	importantly,
	this implies that $y$ is contained in the column span of $A$.
	The leading principal minor of order $n-1$ of $\mathbb{G}(p) = \mathbb{G}(p^{n-1})$ is the determinant of the $(n-1) \times (n-1)$ matrix
	\begin{align*}
		B &=		
		\begin{bmatrix}
			(p^{n-1}_{v_2})^T \\
			\vdots \\
			(p^{n-1}_{v_{n}})^T
		\end{bmatrix}
		\begin{bmatrix}
			p^{n-1}_{v_2} & \cdots & p^{n-1}_{v_{n}}
		\end{bmatrix}\\
		&=
		\begin{bmatrix}
			A^T A & A^T (y+z) \\
			(y+z)^T A & (y+z)^T (y+z)
		\end{bmatrix}\\
		&=
		\begin{bmatrix}
			A^T A & A^T y \\
			y^T A & y^T y
		\end{bmatrix}\\
		&= 
		\begin{bmatrix}
			A^T \\
			y^T
		\end{bmatrix}
		\begin{bmatrix}
			A & y
		\end{bmatrix}.
	\end{align*}
	As $y$ is contained in the column span of $A$,
	the matrix $[A ~ y]$ has rank at most $n-2$.
	However, this implies $\det B = 0$,
	contradicting our leading principal minor assumption for $\mathbb{G}(p)$.
	Hence, $\|z\|^2 \neq 0$.
	
	Fix $x := z/\|z\|^2$.
	Define the $d \times d$ matrix $\tilde{M}$ where $\tilde{M}^T = [ e_1 ~ \cdots ~ e_{n-2} ~ x ~ \mathbf{0} ~ \cdots ~ \mathbf{0}]$.
	We note that the matrix $\tilde{M}$ is an isometry (in the quadratic form sense) from the linear subspace of $\mathbb{C}^d$ spanned by $p^{n-1}_{v_2},\ldots, p^{n-1}_{v_{n}}$ to the linear subspace $\mathbb{C}^{n-1} \times \{0\}^{d-n+1}$.
	By Witt's theorem (see, for example, \cite[Theorem 11.15]{roman}),
	there exists a matrix $M \in O(d, \mathbb{C})$ where $Mp^{n-1}_{v_j} = \tilde{M}p^{n-1}_{v_j}$ for each $j \in \{1,\ldots,n\}$.
	With this we now fix $p^n_v = Mp^{n-1}_v$ for each $v \in V$.
	The result now follows from completing the inductive argument and fixing $q = p^d$.
\end{proof}

The map $\tilde{f}_{G,d}$ allows us to more easily define the cardinality of the set $C_d(G,p)$ for most realisations $p$.

\begin{lemma}\label{l:xgd}
	Let $G=(V,E)$ be a graph with $|V| \geq d+1$,
	and fix a sequence of $d$ vertices $v_1,\ldots, v_d$.
	Then the image of $\tilde{f}_{G,d}$ is Zariski dense in the image of $f_{G,d}$ (and so $\tilde{f}_{G,d}$ is dominant),
	and
	\begin{align*}
		\left|\tilde{f}_{G,d}^{-1}\left(f_{G,d}(p) \right) \right| = 2^{d} |C_d(G,p)|
	\end{align*}
	for almost all\footnote{Here we say that a property $P$ holds \emph{for almost all points in $\mathbb{C}^n$} if there exists a Zariski open subset $U \subset \mathbb{C}^n$ of points where property $P$ holds, in which case we say that any point where property $P$ holds is a \emph{general point (with respect to $P$)}.} $p \in (\mathbb{C}^d)^V$.
\end{lemma}

\begin{proof}
	For every $p \in (\mathbb{C}^d)^V$,
	let $\mathbb{G}(p)$ be the matrix defined in the statement of \Cref{l:rotate}.
	With this, we fix the proper algebraic set
	\begin{align*}
		Z := \Big\{ p \in (\mathbb{C}^d)^V : \text{ $\mathbb{G}(p)$ has a zero leading principal minor, or $p$ does not affinely span $\mathbb{C}^d$}  \Big\}.
	\end{align*}
	Since $Z$ contains hypersurfaces in $(\mathbb{C}^d)^V$ and $Z \neq (\mathbb{C}^d)^V$
	it has dimension $d|V|-1$.
	By \Cref{l:rotate},
	the set $f_{G,d}((\mathbb{C}^d)^V \setminus Z)$ is a subset of the image of $\tilde{f}_{G,d}$.	
	Hence, the image of $\tilde{f}_{G,d}$ is Zariski dense in the image of $f_{G,d}$.
	
	Since $\dim \ell_d(G) = \rank \diff f_{G,d}(p)$ for any general realisation $p\in (\mathbb{C}^d)^V$,
	it follows from \Cref{t:asimowroth} that $G$ is $d$-rigid if and only if $\dim \ell_d(G) = d|V|- \binom{d+1}{2}$.
	Suppose that $G$ is not $d$-rigid.
	The extension of \Cref{t:asimowroth} to complex realisations gives that the set $C_d(G,p)$ contains infinitely many points for almost all realisations $p$;
	in particular,
	it can be used to prove that for almost all $p \in (\mathbb{C}^d)^V \setminus Z$,
	there exists a sequence $(p^n)_{n \in \mathbb{N}}$ of realisations where $p^n$ is not congruent to either $p^m$ or $p$ for each $m \neq n$, and $p^n \rightarrow p$ as $n \rightarrow \infty$ in the standard complex norm for $(\mathbb{C}^d)^V$ (note that this is an actual metric and not the quadratic form $\| \cdot\|^2$).
	An application of \Cref{l:rotate} to both $p$ and each $p^n$ gives that $\tilde{f}_{G,d}^{-1}\left(f_{G,d}(p)\right)$ also contains infinitely many points as is required.
	
	We now suppose that $G$ is $d$-rigid,
	i.\,e., $\dim \ell_d(G) = d|V|- \binom{d+1}{2}$.
	We split the proof into three cases.
	
	\textbf{(Case 1: $f_{G,d}(Z)$ is Zariski dense in $\ell_d(G)$ and $|V| \geq d+2$.)}
	In this case,
	the restriction of the map $f_{G,d}$ to $Z$ and $\ell_d(G)$, which we now denote by $g : Z \rightarrow \ell_d(G)$, is dominant.
	Hence, by \Cref{borel91},
	there exists a non-empty Zariski open subset $U \subset Z$ where 
	\begin{align*}
		\rank \diff g(p) = \dim \ell_d(G) = d|V|- \binom{d+1}{2}
	\end{align*}
	for each $p \in U$.
	Since $\rank \diff g(p) \leq \rank \diff f_{G,d}(p)$,
	we have that $\rank \diff f_{G,d}(p) = d|V| - \binom{d+1}{2}$ for each $p \in U$.
	As $Z$ has dimension $d|V|-1$,
	it contains an irreducible component of dimension $d|V|-1$.
	It follows that the nullity of $\diff g(p)$ is $\binom{d+1}{2} - 1$ for each $p \in U$.	
	If $q$ is congruent to $p \in Z$ then $\mathbb{G}(q) = \mathbb{G}(p)$ and $q$ is affinely spanning if and only if $p$ is,
	hence $Z$ is closed under congruences.
	This implies that for each $p \in U$,
	the map from the affine congruences of $\mathbb{C}^d$ to the corresponding congruent realisation of $p$ in $Z$ is not injective,
	as otherwise the nullity of $\diff g(p)$ would be at least $\binom{d+1}{2}$ for each $p \in U$.
	Hence, any realisation in $U$ cannot affinely span $\mathbb{C}^d$.
	The set of realisations of $G$ that do not affinely span $\mathbb{C}^d$ is the intersection of $|V| - d \geq 2$ pairwise-distinct hypersurfaces,
	and so $U$ is contained in a $(d|V|-2)$-dimensional subvariety in $Z$.
	However, this now contradicts that $U$ is a non-empty Zariski open subset of the $(d|V|-1)$-dimensional algebraic set $Z$.

	\textbf{(Case 2: $f_{G,d}(Z)$ is not Zariski dense in $\ell_d(G)$ and $|V| \geq d+2$.)}
	Fix the non-empty Zariski open set
	\begin{align*}
		X := \left\{ p \in X_{G,d} : \tilde{f}_{G,d}(p) \text{ is not contained in the closure of $f_{G,d}(Z)$ }   \right\}.
	\end{align*}
	Choose any $p \in X$.
	By \Cref{l:rotate}, for each realisation $p' \in (\mathbb{C}^d)^V$ that is equivalent to $p$,
	there exists some other realisation $p'' \in X_{G,d}$ that is congruent to $p'$.
	Hence, it now suffices to prove that $p$ is congruent to exactly $2^d$ realisations (including itself) in $X_{G,d}$.
	
	Choose any $q \in X_{G,d}$ congruent to $p$,
	and let $A \in O(d,\mathbb{C})$ be the matrix that maps the vertices of $p$ to the vertices of $q$.
	As $p,q \in X_{G,d}$, we have 
	\begin{align*}
		\sum_{n=1}^k A_{j,n} [p_{v_{k+1}}]_n = [q_{v_{k+1}}]_j = 0
	\end{align*}
	for each $1 \leq k < j \leq d$,
	hence $A_{j,k}=0$ for each $1 \leq k < j \leq d$.
	Since $A^T$ maps the vertices of $q$ to the vertices of $p$,
	we similarly have $A_{j,k}=0$ for each $1 \leq j < k \leq d$.
	As the vertices of $p$ linearly span $\mathbb{C}^d$,
	the set $\tilde{f}_{G,d}^{-1}\left(f_{G,d}(p) \right)$ is in one-to-one correspondence with the diagonal orthogonal matrices.
	The set of diagonal orthogonal matrices is exactly the set of diagonal matrices $A$ where $A_{j,j} = \pm 1$ for each $j \in \{1,\ldots, d\}$,
	hence there are $2^d$ of them.
	This now concludes the proof.
	
	\textbf{(Case 3: $|V| =d+1$.)}	
	Since $G$ is $d$-rigid, it can be easily verified that $G$ must be the complete graph on $d+1$ vertices.
	Furthermore,
	every element $p \in (\mathbb{C}^d)^V \setminus Z$ is only equivalent to congruent realisations (\cite[Corollary 8]{Gortler2014}),
	and each equivalent realisation is also contained in $(\mathbb{C}^d)^V \setminus Z$.
	We now define $X = X_{G,d} \setminus Z$ and repeat the same technique as was utilised in Case~2.
	This now concludes the proof.
\end{proof}

\subsection{Defining the \texorpdfstring{$d$-realisation}{d-realisation} number}

Before stating our next result, we require the following technical result regarding dominant morphisms.
The proof of this can be seen in \Cref{sec:dom}.

\begin{theorem}\label{t:deg}
	Let $X \subset \mathbb{C}^m$ and $Y \subset \mathbb{C}^n$ be varieties and $f:X \rightarrow Y$ be a dominant morphism.
	Then the following are equivalent:
	\begin{enumerate}
		\item \label{t:deg1} $\dim X = \dim Y$.
		\item \label{t:deg2} $f$ is \emph{generically finite},
		i.e., for a general $y \in Y$,
		the fibre $f^{-1}(y)$ is a finite set.
		\item \label{t:deg3} There exists a $k \in \mathbb{N}$ and a non-empty Zariski open subset $U \subset X$ where $|f^{-1}(f(x))| = k$ for all $x \in U$.
		Furthermore,
		if $x \in U$ and $x' \in f^{-1}(f(x))$,
		then $\rank \diff f(x') = \dim Y$.
	\end{enumerate}
\end{theorem}

\begin{proposition}\label{p:dom}
	Let $G=(V,E)$ be a graph with $|V| \geq d+1$.
	Then the following are equivalent:
	\begin{enumerate}
		\item\label{p:dom1} $G$ is $d$-rigid.
		\item\label{p:dom2} The map $\tilde{f}_{G,d}$ is generically finite.
		\item\label{p:dom3} There exists an $n \in \mathbb{N}$ and a non-empty Zariski open subset $U \subset (\mathbb{C}^d)^V$ where $ |C_d(G,p)|=n$ for all $p \in U$.
		Furthermore,
		if $p \in U$ and $q$ is an equivalent $d$-dimensional realisation of $G$,
		then $\rank \diff f_{G,d}(q) =  d|V| - \binom{d+1}{2}$.
	\end{enumerate}
\end{proposition} 

\begin{proof}
	By \Cref{l:xgd},
	the map $\tilde{f}_{G,d}$ is dominant.	
	Since $\dim \ell_d(G) = \rank \diff f_{G,d}(p)$ for any general realisation $p\in (\mathbb{C}^d)^V$,
	it follows from \Cref{t:asimowroth} that $G$ is $d$-rigid if and only if 
	\begin{align*}
		\dim \ell_d(G) = d|V|- \binom{d+1}{2} = \dim X_{G,d}.
	\end{align*}
	The result now holds by applying \Cref{t:deg} to the map $\tilde{f}_{G,d}$.
\end{proof}

Using \Cref{p:dom} we can now make the following well-defined definition of the $d$-realisation number for any graph.

\begin{definition}\label{def:count}
	The \emph{$d$-realisation number} of a graph $G=(V,E)$ is an element of $\mathbb{N} \cup \{\infty\}$ given by
	\begin{align*}
		c_d(G) :=		
		\begin{cases}
			|C_d(G,p)| \text{ for a general realisation $p \in (\mathbb{C}^d)^V$} &\text{if } |V| \geq d+1, \\
			 1 &\text{if } |V| \leq d \text{ and $G$ is complete}, \\
			 \infty &\text{if } |V| \leq d \text{ and $G$ is not complete}.
		\end{cases}
	\end{align*}
\end{definition}

It follows from \Cref{p:dom} that a graph $G$ is $d$-rigid if and only if $c_d(G) < \infty$.

We can relate the $d$-realisation number back to real realisations so long as we allow for equivalent complex realisations when counting.
We first require the following technical lemma.

\begin{lemma}\label{l:realzar}
	If $U$ is a non-empty Zariski open subset of $\mathbb{C}^n$, then $U \cap \mathbb{R}^n$ is a non-empty Zariski open subset of $\mathbb{R}^n$.
\end{lemma}

\begin{proof}
	For a field $\mathbb{F} \in \{ \mathbb{R}, \mathbb{C}\}$,
	we fix $\mathbb{F}[X^n]$ to be the set of all $n$-variable polynomials over $\mathbb{F}$.
	Define the ideal 
	\begin{align*}
		F := \{ f \in \mathbb{C}[X^n] : f(p) = 0 \text{ for all } p \notin U\}.
	\end{align*}
	Since $U$ is a Zariski open set,
	the zero set
	\begin{align*}
		Z(F) := \{p \in \mathbb{C}^n: f(p) = 0 \text{ for all } f \in F\}
	\end{align*}
	is a Zariski closed set and the complement of $U$.
	As $U$ is non-empty we immediately have that $F \neq \{0\}$.
	For each $f \in F$,
	there exist two unique polynomials $f_{\mathcal{R}},f_{\mathcal{I}} \in \mathbb{R}[X^n]$ such that $f = f_{\mathcal{R}}+ i f_{\mathcal{I}}$:
	to see this, note that there exists a finite subset $A \subset \mathbb{Z}^n_{\geq 0}$ and real values $a_x,b_x $ for each $x \in A$ such that $f(p) = \sum_{x \in A} (a_x + i b_x) p^x$ (here $p^x = p_1^{x_1} \cdots p_n^{x_n}$ for each $x=(x_1,\ldots,x_n)$),
	and so $f_{\mathcal{R}}(p) = \sum_{x \in A} a_x p^x$ and $f_{\mathcal{I}}(p) = \sum_{x \in A} b_x p^x$.
	For each $p \in \mathbb{R}^n$ we note that $f(p) = 0$ if and only if $f_{\mathcal{R}}(p) = f_{\mathcal{I}}(p) = 0$.
	Fix $F_\mathcal{R} := \{ f_{\mathcal{R}} :f \in F\}$ and $F_\mathcal{I} := \{ f_{\mathcal{I}} :f \in F\}$,
	and for every set of (real or complex) polynomials $S \subset \mathbb{C}[X^n]$ fix $Z_{\mathbb{R}}(S):= \{p \in \mathbb{R}^n : f(p) = 0 \text{ for all } f \in S \}$.
	Then $Z_{\mathbb{R}}(F) = Z_{\mathbb{R}}(F_{\mathcal{R}}) \cap Z_{\mathbb{R}}(F_{\mathcal{I}})$.
	Hence, the set $Z_{\mathbb{R}}(F)$ is a Zariski closed subset of $\mathbb{R}^n$ since $Z_{\mathbb{R}}(F_{\mathcal{R}})$ and $Z_{\mathbb{R}}(F_{\mathcal{I}})$ are Zariski closed subsets of $\mathbb{R}^n$.
	Furthermore,
	the set $Z_{\mathbb{R}}(F)$ is a proper Zariski closed subset of $\mathbb{R}^n$:
	this follows from the observing the equivalence
	\begin{align*}
		Z_{\mathbb{R}}(F) = \mathbb{R}^n \quad \Leftrightarrow \quad F_{\mathcal{R}} = F_{\mathcal{I}} = \{0\}  \quad \Leftrightarrow \quad F= \{0\},
	\end{align*}
	with the latter property contradicting our previous observation that $F \neq \{0\}$.
	The result now holds as $U \cap \mathbb{R}^n = \mathbb{R}^n \setminus Z_{\mathbb{R}}(F)$.
\end{proof}

\begin{corollary}\label{cor:realstuff}
	For each graph $G=(V,E)$,
	there exists a non-empty Zariski open subset $U_\mathbb{R} \subset (\mathbb{R}^d)^V$ of real $d$-dimensional realisations $p$ where $|C_d(G,p)| = c_d(G)$.
\end{corollary}

\begin{proof}
	The result is trivial if $G$ is complete and $|V| \leq d$.	
	If $G$ is not $d$-rigid then $c_d(G) = \infty$ and the result holds by \Cref{t:asimowroth}.
	If $G$ is $d$-rigid with $|V| \geq d+1$ then we fix $U \subset (\mathbb{C}^d)^V$ to be the set defined in \Cref{p:dom} and define $U_{\mathbb{R}} := U \cap (\mathbb{R}^d)^V$.
	The set $U_{\mathbb{R}}$ is now a Zariski open subset of $(\mathbb{R}^d)^V$ by \Cref{l:realzar}.
\end{proof}

It is important to note that \Cref{cor:realstuff} does not guarantee that every other equivalent realisation in $C_d(G,p)$ is also real.
In fact we note that this is often not the case.
For example, any minimally $d$-rigid graph with a vertex of degree $d$ must have general $d$-dimensional realisations with complex equivalent realisations:
this realisation can be chosen in a similar way to the right realisation pictured in \Cref{fig:real}.
Failing this,
it is then natural to ask whether each graph $G$ has at least one real generic $d$-dimensional realisation $p$ where every other equivalent realisation in $C_d(G,p)$ is also real.
It was proven in \cite{JacksonOwen} that this too is false when $d=2$.
It does, however, remain open whether the statement is false when $d>2$.
For more on the topic of equivalent real realisations for a graph,
see \cite{Bartzos2021}.


When $d=1$, it is relatively easy to compute the realisation number of a graph.
We first require the following terminology.
A connected graph is \emph{biconnected} if it has at least 2 vertices and the removal of any vertex will always produce a connected graph.
We remark that the set of biconnected graphs is exactly the set of 2-connected graphs with the addition of the complete graph with two vertices.
Given a graph $G=(V,E)$, we say that a subset $F \subset E$ is a \emph{biconnected component} of $G$ if the subgraph $H=(V[F],F)$ induced by $F$ is biconnected and no subgraph of $G$ containing $H$ that is not $H$ itself is biconnected.
The biconnected components of a graph form a partition of its edge set.

\begin{proposition}\label{p:1d}
	Let $G$ be a graph and fix $k$ to be the number of biconnected components of $G$.
	Then $c_1(G) < \infty$ if and only if $G$ is connected.
	If $G$ is connected with at least one edge then $c_1(G) = 2^{k-1}$.
\end{proposition}

\begin{proof}
	If we glue two connected graphs $H_1$ and $H_2$ at exactly one vertex to form a graph $H$ then $c_1(H) = 2c_1(H_1)c_1(H_2)$.
	The result now follows from the observation that $c_1(G) = 1$ if $G$ is biconnected.
\end{proof}

As the number of biconnected components of a graph can be computed in linear time \cite{HopcroftTarjan}, it follows that there exists a linear time algorithm for computing $c_1(G)$.
The current fastest deterministic algorithm for computing $c_2(G)$ when $G$ is minimally 2-rigid runs in exponential time \cite{PlaneCount}.
There are no other known deterministic algorithms for computing $c_d(G)$ outside of restricting to specific families of graphs (e.g., chordal graphs).
%
%
%
%

\section{Counting realisations on a complex sphere}\label{sec:countsphere}

Our aim in this section is to prove that the definition of the spherical $d$-realisation number alluded to in the introduction can be stated in a rigorous manner (see \Cref{def:countsphere}).

\subsection{Complex spherical rigidity map}

We define
\begin{align*}
	\mathbb{S}_{\mathbb{C}}^d := \left\{ x \in \mathbb{C}^{d+1} : \|x\|^2=1 \right\}
\end{align*}
to be the \emph{complexification of the $d$-dimensional (unit) sphere}.
As $\mathbb{S}_{\mathbb{C}}^d$ is a smooth connected variety (and hence irreducible),
so too is $(\mathbb{S}_{\mathbb{C}}^{d})^V$ for any finite set $V$.
Like how we chose to consider a larger set of realisations by switching from $(\mathbb{R}^d)^V$ to $(\mathbb{C}^d)^V$,
we now consider spherical realisations in $\mathbb{S}_\mathbb{C}^d$ by extending the set $(\mathbb{S}^d)^V$ to the set $(\mathbb{S}_{\mathbb{C}}^d)^V$.
Any spherical realisation in $p \in (\mathbb{S}^d)^V$ is now said to be a \emph{real spherical realisation}.
At a spherical realisation $p \in (\mathbb{S}_{\mathbb{C}}^{d})^{V}$,
the linear space
\begin{align*}
	T_p (\mathbb{S}_{\mathbb{C}}^d)^V : =\left\{x \in (\mathbb{C}^{d+1})^V : x_v \cdot p_v=0 \text{ for all } v \in V \right\}
\end{align*}
is the tangent space of $(\mathbb{S}_{\mathbb{C}}^{d})^{V}$ at $p$.

For any graph $G=(V,E)$ we define the \emph{complex spherical rigidity map} to be the map
\begin{align*}
	s_{G,d} \colon (\mathbb{S}_{\mathbb{C}}^{d})^{V} \rightarrow \mathbb{C}^{E}, ~ p \mapsto \left( \frac{1}{2}\|p_v-p_w\|^2 \right)_{vw \in E} =\left( 1 - p_v \cdot p_w\right)_{vw \in E}.
\end{align*}
We denote the Zariski closure of the image of $s_{G,d}$ by $\ell^*_d(G)$.
Since the domain of $s_{G,d}$ is irreducible,
$\ell_d^*(G)$ is a variety.
The derivative of $s_{G,d}$ at a point $p \in (\mathbb{S}_{\mathbb{C}}^{d})^{V}$ can be seen to be the linear map
\begin{align*}
	\diff s_{G,d}(p) \colon T_p (\mathbb{S}_{\mathbb{C}}^{d})^{V} \rightarrow \mathbb{C}^{E}, ~ x \mapsto \left( (p_v-p_w) \cdot (x_v-x_w) \right)_{vw \in E} = \left(-(p_v \cdot x_w + p_w \cdot x_v) \right)_{vw \in E}.
\end{align*}
We define two realisations $p,q \in (\mathbb{S}_{\mathbb{C}}^{d})^V$ to be \emph{congruent} (denoted by $p \sim q$) if and only if $p = Aq$ for some $A \in O(d+1, \mathbb{C})$.
If the vertex set of $p$ linearly spans $\mathbb{C}^{d+1}$,
then two realisations $p,q$ are congruent if and only if $s_{K_V,d}(p) = s_{K_V,d}(q)$.
For all $p \in (S_\mathbb{C}^{d})^V$,
we define $C^*_d(G,p) := s^{-1}_{G,d} (s_{G,d}(p))/_\sim$ to be the \emph{spherical realisation space of $(G,p)$}.

We can link the infinitesimal properties of rigidity in $\mathbb{C}^d$ and in $\mathbb{S}^d_\mathbb{C}$ with the following result.

\begin{lemma}\label{l:equivrank}
	Let $p \in (\mathbb{C}^{d})^V$ and $q \in (\mathbb{S}_{\mathbb{C}}^{d})^V$ be such that $[q_{v}]_{d+1} \neq 0$ and $[p_v]_i = [q_v]_i/[q_v]_{d+1}$ for all $v \in V$ and $i \in \{1,\ldots,d\}$.
	Then $\rank \diff f_{G,d}(p)= \rank \diff s_{G,d}(q)$.
\end{lemma}

\begin{proof}
	Define the bijective linear map
	\begin{align*}
		\varphi \colon T_q (\mathbb{S}_{\mathbb{C}}^{d})^{V} \rightarrow (\mathbb{C}^{d})^V, ~ x \mapsto \left( \frac{[x_v]_1}{[q_v]_{d+1}},\ldots, \frac{[x_v]_d}{[q_v]_{d+1}} \right)_{v \in V}.
	\end{align*}
	For any $u \in T_q (\mathbb{S}_{\mathbb{C}}^{d})^{V}$ we have
	\begin{eqnarray*}
		&~& (p_v-p_w)\cdot (\varphi(u)_v- \varphi(u)_w)\\
		&=& \sum_{i=1}^d \frac{[q_v]_i[u_v]_i}{[q_v]^2_{d+1}} +\frac{[q_w]_i[u_w]_i}{[q_w]^2_{d+1}} - \frac{[q_v]_i[u_w]_i + [q_w]_i[u_v]_i}{[q_v]_{d+1}[q_w]_{d+1}} \\
		&=& -\frac{[q_v]_{d+1}[u_v]_{d+1}}{[q_v]^2_{d+1}}-\frac{[q_w]_{d+1}[u_w]_{d+1}}{[q_w]^2_{d+1}} + \frac{[q_v]_{d+1}[u_w]_{d+1} +[q_w]_{d+1}[u_v]_{d+1}}{[q_v]_{d+1}[q_w]_{d+1}} - \frac{q_v\cdot u_w + q_w \cdot u_v}{[q_v]_{d+1}[q_w]_{d+1}} \\
		&=& - \frac{q_v \cdot u_w + q_w \cdot u_v}{[q_v]_{d+1}[q_w]_{d+1}}.
	\end{eqnarray*}
		It follows that $u \in \ker \diff s_{G,d}(q)$ if and only if $\varphi(u) \in \ker \diff f_{G,d}(p)$.
		As $\varphi$ is bijective,
		we have $\rank \diff f_{G,d}(p)= \rank \diff s_{G,d}(q)$ as required.
\end{proof}

From \Cref{l:equivrank} we know that a graph is rigid (respectively, independent) in $\mathbb{R}^d$ if and only if it is rigid (respectively, independent) in $\mathbb{S}^d$.
Due to this, we can obtain an analogue of \Cref{l:domind} for the map $s_{G,d}$.

\begin{lemma}\label{l:domindsphere}
	Let $G=(V,E)$ be any graph.
	Then the following are equivalent:
	\begin{enumerate}
		\item \label{l:sdomind1} $G$ is $d$-independent.
		\item \label{l:sdomind2} The map $s_{G,d}$ is dominant.	
		\item \label{l:sdomind3} $\ell_d^*(G) = \mathbb{C}^E$.
	\end{enumerate}
\end{lemma} 

\begin{proof}
	By the definition of a dominant map, \ref{l:sdomind2} and \ref{l:sdomind3} are equivalent.
	Hence, we only need to prove that \ref{l:sdomind1} and \ref{l:sdomind2} are equivalent.
	By \Cref{l:domind},
	we only need to show that $s_{G,d}$ is dominant if and only if $f_{G,d}$ is dominant.
	This is equivalent to proving the following:
	there exists $p \in (\mathbb{C}^{d})^V$ with $\rank \diff f_{G,d}(p) = |E|$ if and only if
	there exists a $q \in (\mathbb{S}_{\mathbb{C}}^{d})^V$ with $\rank \diff s_{G,d}(p) = |E|$ (see \Cref{borel91}).
	Hence, the result holds by \Cref{l:equivrank}.
\end{proof}

\subsection{Defining the spherical \texorpdfstring{$d$-realisation}{d-realisation} number}

Let $G=(V,E)$ be a graph with at least $d+1$ vertices and fix a sequence of $d$ vertices $v_1,\ldots, v_d$.
We now define the algebraic set
\begin{align}\label{eq:yset}
	Y_{G,d} := \left\{ p \in (\mathbb{S}^{d}_\mathbb{C})^V : [p_{v_k}]_j=0 \text{ if } j \geq k + 1 \text{ and }  [p_{v_1}]_1=1 \right\}
\end{align}
Since $Y_{G,d}$ is a connected smooth manifold,
it is irreducible.
We further note that $Y_{G,d}$ has dimension $d|V| - \binom{d+1}{2}$.
With this, we define the morphism
\begin{align*}
	\tilde{s}_{G,d} : Y_{G,d} \rightarrow \ell^*_d(G), ~ p \mapsto s_{G,d}(p),
\end{align*}
i.\,e., the restriction of $s_{G,d}$ to the domain $Y_{G,d}$ and the codomain $\ell^*_d(G)$.
Since every $d$-dimensional realisation of $G$ is equivalent to at least one element of $Y_{G,d}$ it follows that $\tilde{s}_{G,d}$ is a dominant map.

The methods utilised in \Cref{l:xgd} can easily be adapted to work for realisations in the $d$-dimensional sphere.

\begin{lemma}\label{l:ygd}
	Let $G=(V,E)$ be a graph with $|V| \geq d+1$,
	and fix a sequence of $d$ vertices $v_1,\ldots, v_d$.
	Then the image of $\tilde{s}_{G,d}$ is Zariski dense in the image of $s_{G,d}$ (and so $\tilde{s}_{G,d}$ is dominant),
	and
	\begin{align*}
		\left|\tilde{s}_{G,d}^{-1}\left(s_{G,d}(p) \right) \right| = 2^{d} |C^*_d(G,p)|
	\end{align*}
	for almost all $p \in (\mathbb{S}_\mathbb{C}^d)^V$.
\end{lemma}

The spherical version of \Cref{p:dom} is also proved in an analogous way.

\begin{proposition}\label{p:domsphere}
	Let $G=(V,E)$ be a graph with $|V| \geq d+1$.
	Then the following are equivalent:
	\begin{enumerate}
		\item\label{p:domsphere1} $G$ is $d$-rigid.
		\item\label{p:domsphere2} The map $\tilde{s}_{G,d}$ is generically finite.
		\item\label{p:domsphere3} There exists an $n \in \mathbb{N}$ and a non-empty Zariski open subset $U \subset (\mathbb{S}_\mathbb{C}^d)^V$ where $ |C^*_d(G,p)|=n$ for all $p \in U$.
		Furthermore,
		if $p \in U$ and $q$ is an equivalent $d$-dimensional spherical realisation of $G$,
		then $\rank \diff s_{G,d}(q) =  d|V| - \binom{d+1}{2}$.
	\end{enumerate}
\end{proposition} 

Using \Cref{p:domsphere} we can now make the following well-defined definition of the spherical $d$-realisation number for a minimally $d$-rigid graph.

\begin{definition}\label{def:countsphere}
	The \emph{spherical $d$-realisation number} of a graph $G=(V,E)$ is an element of $\mathbb{N} \cup \{\infty\}$ given by
	\begin{align*}
		c^*_d(G) :=		
		\begin{cases}
			|C^*_d(G,p)| \text{ for a general realisation $p \in (\mathbb{S}_\mathbb{C}^d)^V$} &\text{if } |V| \geq d+1, \\
			 1 &\text{if } |V| \leq d \text{ and $G$ is complete}, \\
			 \infty &\text{if } |V| \leq d \text{ and $G$ is not complete}.
		\end{cases}
	\end{align*}
\end{definition}

It follows from \Cref{p:domsphere} that a graph $G$ is $d$-rigid if and only if $c_d(G) < \infty$.
It is worth noting that,
while $c^*_d(G) < \infty$ if and only if $c_d(G) < \infty$,
this does not mean the two numbers are always equal.
This can be seen in the following example.

\begin{example}
  We consider the 3-prism graph.
  It is minimally 2-rigid with six vertices, all of which have degree three.
  For this graph we have $c_2(G)=12$, which can be computed with the algorithm described in \cite{PlaneCount}.
  In fact, there exist edge-length assignments for which we also obtain exactly 12 real realisations (see \Cref{fig:threeprismplane}).
  However, the same graph embedded on the sphere gives $c_2^*(G)=16$,
  which can be computed by the algorithm described in \cite{SphereCount}.
  Again there exist edge-length assignments such that each realisation is real (see \Cref{fig:threeprismsphere}).
  (Both algorithms use the alternative definitions for realisation numbers and so the outputted numbers first need to be halved.)
\end{example}

The natural analogue of \Cref{cor:realstuff} also holds for spheres by a similar method.

\begin{corollary}\label{cor:sphrealstuff}
	For each graph $G=(V,E)$,
	there exists a non-empty Zariski open subset $U_\mathbb{R} \subset (\mathbb{S}^d)^V$ of real $d$-dimensional spherical realisations $p$ where $|C^*_d(G,p)| = c^*_d(G)$.
\end{corollary}

\begin{remark}
	It is currently open whether every graph $G$ has a real generic $d$-dimensional spherical realisation that is equivalent (modulo congruences) to exactly $c_d^*(G)$ real spherical realisations.
	The counter-example used in the previous section for real realisations in the plane required the existence of a graph $H$ where $c_2(H)$ is odd.
	Similarly, any graph $G$ where $c^*_2(G)$ is odd would be a counter-example to the result, however no such graph has yet been found.
\end{remark}

We finish the section by observing that the spherical realisation number is always equal to the realisation number when $d=1$,
which can be proven using the same method as \Cref{p:1d}.

\begin{proposition}\label{p:1dsame}
	For any graph $G$ we have $c_1^*(G) =c_1(G)$.
\end{proposition}

\begin{remark}\label{rem:hyperbolic}
	All of the results in this section can be extended from the sphere to the light cone of any pseudo-Euclidean space.
	To be specific,
	we define for each $1 \leq k \leq d$ the quadratic form
	\begin{align*}
		q_{d,k}: \mathbb{C}^{d+1} \rightarrow \mathbb{C}, ~ (x_0,\ldots,x_{d}) \mapsto \sum_{j=0}^k x_j^2 - \sum_{j=k+1}^d x_j^2,
	\end{align*}
	and we define the light cone
	\begin{align*}
		\mathbb{H}_{d,k} := \left\{ x \in \mathbb{C}^{d+1} : q_{d,k}(x) = 0 \right\}.
	\end{align*}
	The map
	\begin{align*}
		T:\mathbb{H}_{d,k} \rightarrow \mathbb{H}_{d,d}, ~ (x_0,\ldots,x_d) \mapsto (x_0,\ldots, x_k, ix_{k+1},\ldots,x_d)
	\end{align*}
	is an isometry (in the sense that $q_{d,d}(T(x) - T(y)) = q_{d,k}(x-y)$),
	and the space $\mathbb{H}_{d,d}$ is isomorphic to $\mathbb{S}_{\mathbb{C}}^d$.
	Hence, the realisation number of a graph $G$ in any space $\mathbb{H}_{d,k}$ is equal to $c^*_d(G)$.
	Notably,
	the hyperbolic realisation number of a graph (equivalently, the space $\mathbb{H}_{d,d-1}$) is equal to the spherical realisation number of a graph. 
\end{remark}

\section{Coning and the spherical realisation number}\label{sec:cone}

In this section we prove \Cref{t:conecount} by observing that the number of realisations of a framework is projectively invariant (\Cref{kl:proj}).
We then prove that $c_d(G*o) = c_d^*(G*o)$ for any coned graph $G*o$ (\Cref{t:conecount2}).

\subsection{Proof of \texorpdfstring{\Cref{t:conecount}}{Theorem 1.2}}

Throughout this subsection we fix the vertices used to define our various varieties in the following way.
Let $G=(V,E)$ be a graph with $|V| \geq d+1$ and let $G*o$ be its cone.
We fix the vertices $v_1,\ldots, v_d \in V$ to define the set $X_{G,d}$ as given in \Cref{eq:xset}.
From this, we fix the set
\begin{align*}
	X_{G*o,d+1} = \left\{ p' \in (\mathbb{C}^{d+1})^{V*o} : p'_o = \mathbf{0}, ~  [p'_{v_k}]_{j+1}=0 \text{ if } k \leq j \leq d  \right\}\,.
\end{align*}
(This is equivalent to defining the set $X_{G*o,d+1}$ as in \Cref{eq:xset} with the vertices $w_1,\ldots,  w_{d+1}$ by setting $w_1 = o$ and $w_{j+1} = v_j$ for all $j =1,\ldots, d$.)
We also assume that the fixed vertices $v_1,\ldots,v_d$ used to define $X_{G,d}$ and $Y_{G,d}$ (see \Cref{eq:yset}) are always identical.

\begin{lemma}\label{kl:proj}
	Given a graph $G=(V,E)$ with $|V| \geq d+1$,
	choose any $p \in Y_{G,d}$ and $(r_v)_{v \in V} \in (\mathbb{C}\setminus \{0\})^{V}$.
	Define $p' \in X_{G*o,d+1}$ with $p'_v = (r_vp_v,0)$ for each $v \in V$ and $p'_o=\mathbf{0}$.
	Then
	\begin{align*}
		\left| \tilde{f}_{G*o,d+1}^{-1}\left(\tilde{f}_{G*o,d+1}(p') \right) \right| = 2 \left| \tilde{s}_{G,d}^{-1}\Big(\tilde{s}_{G,d}(p) \Big) \right|.
	\end{align*} 
\end{lemma}

\begin{proof}
	Define $\lambda := \tilde{s}_{G,d}(p)=\left( 1 - p_v \cdot p_w\right)_{vw\in E}$ and $\lambda' := \tilde{f}_{G*o,d+1}(p')$.
	We first note that for any $vw \in E$ we have
	\begin{align}\label{eq:lambda}
		\lambda_{vw}' = \frac{1}{2}\|p'_v- p'_w\|^2 = \frac{1}{2} \|p'_v\|^2 + \frac{1}{2} \|p'_w\|^2 - p'_v \cdot p'_w = \lambda'_{ov} + \lambda'_{ow} + (r_vr_w)(\lambda_{vw}-1).
	\end{align}
	Define $S \subset \tilde{f}_{G*o,d+1}^{-1}(\lambda')$ to be the set of realisations $q' \in X_{G*o,d+1}$ where $[q'_{v_1}]_1=r_{v_1}$.
	Note that for each $q' \in \tilde{f}_{G*o,d+1}^{-1}(\lambda')$ we have that
	\begin{align*}
		[q'_{v_1}]_1^2 = \|q'_{v_1} \|^2 = \|q'_{v_1} -q'_o \|^2 = \|p'_{v_1} -p'_o \|^2 = \|r_{v_1} p_{v_1}\|^2 = r_{v_1}^2;
	\end{align*}
	the first equality follows from $q' \in X_{G*o,d+1}$,
	the second from $ov_1 \in E*o$,
	the third from $ \tilde{f}_{G*o,d+1}(q' ) = \tilde{f}_{G*o,d+1}(p')$,
	the fourth from the construction of $p'$ from $p$,
	and the last from $p \in Y_{G,d}$.
	Hence, if $q' \in \tilde{f}_{G*o,d+1}^{-1}(\lambda') \setminus S$ then we have $[q'_{v_1}]_1 = -r_{v_1}$,
	and so $-q' \in S$,
	since 
	As $S \cap (-S) = \emptyset$ and $S \cup (-S) = \tilde{f}_{G*o,d+1}^{-1}(\lambda')$,
	we have that $|\tilde{f}_{G*o,d+1}^{-1}(\lambda')| = 2|S|$ (note that this cardinality need not be finite).
	It now suffices to prove that $|\tilde{s}_{G,d}^{-1}(\lambda)| = |S|$.
	
	Choose any $q \in \tilde{s}_{G,d}^{-1}(\lambda)$ and define $q'\in X_{G*o,d+1}$ by setting $q'_o=\mathbf{0}$ and $q'_v = r_v q_v$ for each $v \in V$.
	It is immediate that $[q'_{v_1}]_1 = r_{v_1}$ and $\frac{1}{2}\|q'_v\|^2 = \lambda'_{ov}$.
	For each $vw \in E$ we have
	\begin{align*}
		 \frac{1}{2}\|q'_v- q'_w\|^2 &=  \frac{1}{2}\|q'_v\|^2 +  \frac{1}{2}\|q'_w\|^2 - q'_v\cdot q'_w \\
		&= \lambda'_{ov} + \lambda'_{ow} + (r_vr_w)(\lambda_{vw}-1) \\
		&\stackrel{_\eqref{eq:lambda}}{=} \lambda'_{vw},
	\end{align*}
	and so $q' \in S$.
	Hence, $|\tilde{s}_{G,d}^{-1}(\lambda)| \leq |S|$.
	
	Now choose any $q' \in S$ and define $q \in (\mathbb{C}^{d+1})^V$ by setting $q_v = q'_v/r_v$ for each $v\in V$.
	We first note that $\|q_v\|^2 = \|q'_v - q'_o\|^2/r^2_v = 1$ for each $v \in V$,
	hence $q \in Y_{G,d}$.
	For each $vw \in E$ we have
	\begin{align*}
		q_v\cdot q_w &= \frac{q'_v\cdot q'_w}{r_vr_w} \\
		&= \frac{\|q'_v\|^2 + \|q'_w\|^2 - \|q'_v-q'_w\|^2}{2 r_vr_w} \\
		&= \frac{\lambda'_{ov} + \lambda'_{ow} - \lambda'_{vw}}{r_vr_w} \\
		&\stackrel{_\eqref{eq:lambda}}{=} 1 - \lambda_{vw},
	\end{align*}
	and so $q \in \tilde{s}_{G,d}^{-1}(\lambda)$.
	Hence, $|S |\leq |\tilde{s}_{G,d}^{-1}(\lambda)|$.
\end{proof}

A nice corollary of \Cref{kl:proj} is that the set $C_{d+1}(G*o,p)$ is projectively invariant,
in the sense that projecting each point $p_v$ along the complex line through the points $\{p_v,p_o\}$ by a non-zero complex scalar does not alter the number of equivalent realisations.

We are now ready to prove our first main theorem.

\begin{proof}[Proof of \Cref{t:conecount}]
	If $G=(V,E)$ is not $d$-rigid then the result follows from \Cref{t:cone}.
	If $|V| \leq d$ and $G$ is complete then $|V*o| \leq d+1$ and $G*o$ is complete,
	and so $c_{d+1}(G*o) = c^*_d(G)$.
	Suppose that $G=(V,E)$ is $d$-rigid with $|V| \geq d+1$.	
	As $G$ is $d$-rigid, the coned graph $G*o$ is $(d+1)$-rigid by \Cref{t:cone}.
	By \Cref{l:xgd} and \Cref{p:dom},
	there exists a Zariski open set $U' \subset (\mathbb{C}^{d+1})^{V*o}$ such that
	\begin{align}\label{e:conenumber}
		\left| \tilde{f}_{G*o,d+1}^{-1} \left(\tilde{f}_{G*o,d+1}(p') \right) \right| = 2^{d+1} c_{d+1}(G*o)
	\end{align}
	for all $p' \in U'$.
	Similarly, it follows from \Cref{l:ygd} and \Cref{p:domsphere} that there exists a Zariski open set $U \subset (\mathbb{S}_\mathbb{C})^{d}$ such that
	\begin{align}\label{e:spherenumber}
		\left| \tilde{s}_{G,d}^{-1} \Big(\tilde{s}_{G,d}(p) \Big) \right| = 2^{d} c_{d}^*(G)
	\end{align}
	for all $p \in U$.
	Fix $U^* := U \times (\mathbb{C} \setminus \{0\})^V$.
	Since $U^*$ is a Zariski open subset of the variety $Y_{G,d} \times \mathbb{C}^V$,
	it is an open and dense subset of $Y_{G,d} \times \mathbb{C}^V$ with respect to the metric topology.
	Similarly,
	since $U'$ is a Zariski open subset of the variety $ X_{G *o,d+1}$,
	it is an open and dense subset of $X_{G *o,d+1}$ with respect to the metric topology.
	
	Define the surjective (and hence dominant) morphism 
	\begin{align*}
		\phi : Y_{G,d} \times \mathbb{C}^V \rightarrow X_{G *o,d+1},
	\end{align*}
	where, given $p'= \phi(p,r)$,
	we have $p'_v=r_vp_v$ for all $v \in V$ and $p'_o=\mathbf{0}$.
	As $\phi$ is dominant and $\dim Y_{G,d} \times \mathbb{C}^V = \dim X_{G *o,d+1}$,
	it follows from \Cref{t:deg} and the inverse mapping theorem for holomorphic maps (see \cite[Theorem 7.5]{FG02}) that the image of any open dense subset of $Y_{G,d} \times \mathbb{C}^V$ (with respect to the metric topology) contains an open subset of $X_{G*o,d+1}$ (with respect to the metric topology).
	Hence, the set $\phi( U^*)$ has a non-empty interior with respect to the metric topology of $X_{G*o,d+1}$.
	Since $U'$ is an open dense subset of $X_{G*o,d+1}$ with respect to the metric topology,
	$\phi( U^*) \cap U'$ contains a realisation $p'$.
	Choose $p \in U$ and $r \in (\mathbb{C}\setminus \{0\})^V$ such that $\phi(p,r) = p'$.
	By \Cref{e:conenumber,e:spherenumber} and \Cref{kl:proj},
	we have that
	\begin{align*}
		2^{d+1} c_{d+1}(G*o) =\left| \tilde{f}_{G*o,d+1}^{-1}\left(\tilde{f}_{G*o,d+1}(p') \right) \right| = 2 \left| \tilde{s}_{G,d}^{-1}\Big(\tilde{s}_{G,d}(p) \Big) \right|= 2^{d+1} c^*_{d}(G),
	\end{align*}
	and so $c^*_{d}(G) = c_{d+1}(G*o)$.
\end{proof}

\subsection{Repetitive coning stabilises realisation numbers}

Our previous techniques of the section can be repurposed to prove the following interesting result.

\begin{theorem}\label{t:conecount2}
	Let $d$ be a positive integer and let $G*o$ be a coning of a graph $G$.
	Then $c_{d}(G*o) = c_{d}^*(G*o)$.
\end{theorem}

\begin{proof}
	If $G$ has less than $d$ vertices then the result is trivial,
	hence we may suppose that $G$ has at least $d$ vertices.
	Let $H = (V',E')$ be the graph formed from $G*o$ by coning again,
	with $V' = V*o \cup \{o'\} = V \cup \{o,o'\}$.
	By \Cref{t:conecount},
	it suffices to prove that $c_d(G*o) = c_{d+1}(H)$.
	For our varieties $X_{G*o,d}$ and $X_{H,d+1}$,
	we fix $v_1 = o$ and choose $d-1$ other vertices $v_2,\ldots,v_d \in V$ for $G*o$,
	and we fix $w_1 = o'$, $w_2 = o$ and $w_{j+1} = v_j$ for each $j \in \{2,\ldots,d\}$ for $H$.
	
	For any realisation $p \in X_{G*o,d}$ of $G*o$,
	define the realisation $p' \in X_{H,d+1}$ where $p'_o = (1,0,\ldots,0)$ and $p'_v = (\frac{1}{2},p_v)$ for each $v \in V$.
	Choose any $\bar{p} \in X_{H,d+1}$ that is equivalent to $p'$.
	Since $o$ and $o'$ are adjacent in $H$,
	we have that $\bar{p}_o = \pm p'_o = (\pm 1,0,\ldots,0)$.
	Suppose that $\bar{p}_o = p'_o= (1,0,\ldots,0)$.
	For each $v \in V$ we have
	\begin{align*} 
		\left\| p'_v-p'_{o} \right\|^2 -  \left\|p'_v-p'_{o'} \right\|^2 =  \left([p'_v]_1 - 1\right)^2 -  [p'_v]_1^2 = 0,
	\end{align*}
	and so
	\begin{align*} 
		0 = \left\|\bar{p}_v-\bar{p}_{o} \right\|^2 - \left\|\bar{p}_v-\bar{p}_{o'} \right\|^2 = \left([\bar{p}_v]_1 - 1 \right)^2 - [\bar{p}_v]_1^2  = -2[\bar{p}_v]_1 +1.
	\end{align*}
	Hence, $[\bar{p}_v]_1 = 1/2$ for all $v \in V$ also.
	It now follows that there exists a realisation $q \in \tilde{f}_{G*o,d}^{-1}(\tilde{f}_{G*o,d}(p))$ such $q' = \bar{p}$.
	Hence, there exists a bijection from $\tilde{f}_{G*o,d}^{-1}(\tilde{f}_{G*o,d}(p))$ to the subset of realisations in $\tilde{f}_{H,d+1}^{-1}(\tilde{f}_{H,d+1}(p'))$ with the vertex $o$ placed at $(1,0,\ldots,0)$.
	As the remaining realisations can be obtained by a reflection in the hyperplane normal to $(1,0,\ldots,0)$,
	we have $|\tilde{f}_{H,d+1}^{-1}(\tilde{f}_{H,d+1}(p'))| = 2|\tilde{f}_{G*o,d}^{-1}(\tilde{f}_{G*o,d}(p))|$ for every $p \in X_{G,d}$.
	
	Fix the set
	\begin{align*}
		Z := \Big\{ p' \in X_{H,d+1} :	p'_o = (1,0,\ldots,0), ~ p'_v = (1/2,p_v) \text{ for each $v \in V$}  \Big\}
	\end{align*}		
	and define the bijective map $\psi : X_{G*o,d} \rightarrow Z$ that maps $p$ to its unique realisation $p'$ as defined above.
	Next,
	fix the dominant map $\phi : Z \times \mathbb{C}^{V*o} \rightarrow X_{H,d+1}$ where,
	given $p'' = \phi(p',r)$, we have $p'_v = r_v p_v$ for all $v \in V *o$.
	Now choose a non-empty Zariski set $U \subset X_{G*o,d}$ where 
	\begin{align*}
		|\tilde{f}_{G*o,d}^{-1}(\tilde{f}_{G*o,d}(p))| = 2^{d}c_d(G*o)
	\end{align*}
	for each $p \in U$ (\Cref{l:xgd}).
	Note that the set $\psi(U)$ is a non-empty Zariski open subset of $Z$.
	By our previous work we have that for each $p' \in \psi(U)$ we have
	\begin{align*}
		|\tilde{f}_{H,d+1}^{-1}(\tilde{f}_{H,d+1}(p'))| = 2|\tilde{f}_{G*o,d}^{-1}(\tilde{f}_{G*o,d}(p))| = 2^{d+1} c_d(G*o).
	\end{align*}
	Since the set $\psi(U) \times (\mathbb{C} \setminus \{0\})^{V*o}$ is a non-empty Zariski open set,
	its image under the dominant map $\phi$ is a Zariski dense subset of $X_{H,d+1}$.
	It now follows from \Cref{kl:proj} that there exists a Zariski dense subset $U' \subset X_{H,d+1}$ where for each $q \in U'$ we have
	\begin{align*}
		 |\tilde{f}_{H,d+1}^{-1}(\tilde{f}_{H,d+1}(q))| = 2^{d+1} c_d(G*o).
	\end{align*}
	Hence, by \Cref{l:xgd} we have $c_{d+1}(H) = c_d(G*o)$ as required.
\end{proof}

By combining \Cref{t:conecount,t:conecount2},
we obtain the following immediate corollary.

\begin{corollary}\label{cor:conecount2}
	Let $H$ be a graph formed from a graph $G$ by performing $k>0$ coning operations.
Then $c_{d+k}(H) = c_{d+k}^*(H) = c_d^*(G)$.
\end{corollary}

\section{Proof of \texorpdfstring{\Cref{t:sphereplane}}{Theorem 1.1}}\label{sec:flat}

Using \Cref{t:conecount},
we can prove \Cref{t:sphereplane} by instead proving that $c_d(G) \leq c_{d+1}(G*o)$.
Because of this, we need to be able to evaluate $d$-dimensional realisations of $G$ and $(d+1)$-dimensional realisations of $G*o$ at the same time.
One way of considering this is by fixing the cone vertex at the point at infinity so that the distance constraints between the cone vertex and all other vertices becomes a set of linear constraints forcing the vertices into a $d$-dimensional hyperplane.
With this general idea in mind,
we begin to construct such a space and the resulting variant of the rigidity map that comes from it.

We begin with the following prototype function that we will improve to give our required map.
For any graph $G$ with vertex $u$,
fix $h_u$ to be the morphism
\begin{align*}
	h_u : (\mathbb{C}^{(d+1)})^V \times \mathbb{C} \rightarrow \mathbb{C}^{V \setminus \{u\}}, ~ (p,r) \mapsto \left( \frac{r}{2}\left\|p_v \right\|^2 -  [p_v]_{d+1} \right)_{v \in V \setminus \{u\}}.
\end{align*}
The map $h_u$ can be used to check for equivalent realisations for the coned graph when we invert the last coordinate of the cone vertex.

\begin{lemma}\label{l:equivmap}
	Let $p,q$ be two realisations of $G*o$ in $\mathbb{C}^{d+1}$ with $p_{u}= q_{u}=\mathbf{0}$ for some $u \in V$,
	$[p_o]_j=[q_o]_j=0$ for each $j \in \{1,\ldots,d\}$ and $[p_o]_{d+1} = [q_o]_{d+1} \neq 0$.
	Given $r := 1/[p_o]_{d+1}$,
	then $f_{G*o,d+1}(p)= f_{G*o,d+1}(q)$ if and only if $f_{G,d+1}(p|_V) = f_{G,d+1}(q|_V)$ and $h_u(p|_V,r)=h_u(q|_V,r)$.
\end{lemma}

\begin{proof}
	For each $v \in V$ we have
	\begin{eqnarray*}
		\frac{1}{2}\|p_v-p_o\|^2 &=& \frac{1}{2}\|p_v\|^2 - p_v \cdot p_o + \frac{1}{2}\| p_o\|^2\\
		&=& \frac{1}{2}\|p_v\|^2 - \frac{1}{r} [p_v]_{d+1} + \frac{1}{2 r^2} \\
		&=& \frac{1}{r} h_u(p|_V, r)_v  +  \frac{1}{2 r^2},
	\end{eqnarray*}
	and similarly $\|q_v-q_o\|^2 = \frac{1}{r} h_{u}(q|_V, r)_v  +  \frac{1}{2r^2}$.
	Hence, $f_{G*o,d+1}(p)= f_{G*o,d+1}(q)$ if and only if $f_{G,d+1}(p|_V) = f_{G,d+1}(q|_V)$ and $h_u(p|_V,r)=h_u(q|_V,r)$.
\end{proof}

For the remainder of the section we fix $G=(V,E)$ to be a $d$-rigid graph with $|V| \geq d+1$,
and we also fix the distinct vertices $v_1,\ldots, v_d$ of $G$.
We first need to reformat $X_{G*o,d+1}$.
Define the $((d+1)(|V*o|) - \binom{d+2}{2})$-dimensional linear space
\begin{align*}
	X_{G*o,d+1}' := \left\{  p \in (\mathbb{C}^{d+1})^{V*o} :  [p_{v_1}]_{d+1} = 0, ~ [p_o]_k = 0, ~[p_{v_k}]_j = 0 \text{ for all }  1 \leq k \leq j \leq d \right\}.
\end{align*}
Importantly,
this space forces the vertex $v_1$ to lie on the origin and the cone vertex $o$ to lie on the $(d+1)$-axis.
We can link our new space $X_{G*o,d+1}'$ with our previously used space $X_{G*o,d+1}$ (see \cref{eq:xset}) by the bijective linear map $L :X_{G*o,d+1}' \rightarrow X_{G*o,d+1}$ where, given $q=L(p)$,
we have $[q_v]_1 = [p_v]_{d+1} - [p_o]_{d+1}$ and $[q_v]_j = [p_v]_{j-1} - [p_o]_{d+1}$ for each $j \in \{2,\ldots, d+1\}$.
Hence, any result that uses the space $X_{G*o,d+1}$ can be easily replaced by a result that uses the space $X'_{G*o,d+1}$;
for example,
general realisations in $X'_{G*o,d+1}$ have identical properties to general realisations in $X_{G*o,d+1}$.

Next,
we define the $((d+1)|V| - \binom{d+2}{2} + d)$-dimensional linear space
\begin{align*}
	Z_{G,d+1} := \left\{  p \in (\mathbb{C}^{d+1})^{V} :  [p_{v_1}]_{d+1} = 0, ~[p_{v_i}]_j = 0 \text{ for all } i, j \in \{1,\ldots, d\}, ~ j \geq i \right\}.
\end{align*}
Note that the space $X_{G,d}$ embeds into $Z_{G,d+1}$ under the injective linear map 
\begin{align*}
	\lambda: X_{G,d} \rightarrow Z_{G,d+1}
\end{align*}
where, given $q= \lambda(p)$ and a vertex $v \in V$,
we have that $[q_v]_j = [p_v]_j$ for each $j \in \{1,\ldots,d\}$ and $[q_v]_{d+1} = 0$.
Let 
\begin{align*}
	\phi: Z_{G,d+1} \times (\mathbb{C}\setminus \{0\}) \rightarrow X_{G*o,d+1}'
\end{align*}
be the injective open map where,
given $\phi(p,r) = q$,
we have $q_v = p_v$ for all $v \in V$ and $q_o = (0 , \ldots, 0, 1/r)$.
The only realisations in $X_{G*o,d+1}'$ not contained in the image of $\phi$ are those of the form $p$ with $[p_o]_{d+1} = 0$.

With all of these spaces defined,
we are now ready to define our main morphism:
\begin{align*}
	g : Z_{G,d+1} \times \mathbb{C} \rightarrow \mathbb{C}^E \times \mathbb{C}^{V \setminus \{v_1\}} \times \mathbb{C} , ~ (p,r) \mapsto \left( f_{G,d+1}(p) , h_{v_1}(p,r), r \right).
\end{align*}
As proven by our next result,
the map $g$ is effectively identical in behaviour to $\tilde{f}_{G*o,d+1}$ over the realisations in the image of $\phi$.

\begin{lemma}\label{l:gmap}
	For any $p,q \in X_{G*o,d+1}'$ where $[p_o]_{d+1}, [q_o]_{d+1} \neq 0$,
	the following are equivalent:
	\begin{enumerate}
		\item $f_{G*o,d+1} (p) = f_{G*o,d+1} (q)$.
		\item $g \circ \phi^{-1}(p) = g \circ \phi^{-1}(q)$, or $g \circ \phi^{-1}(p) = g \circ \phi^{-1}(-q)$.
	\end{enumerate}
	Hence, $|\tilde{f}_{G*o,d+1}^{-1}(f_{G*o,d+1} (p))| = 2|g \circ \phi^{-1}(p)|$.
\end{lemma}

\begin{proof}
	It follows from our construction of $X_{G*o,d+1}'$ that $[p_o]_{d+1} = \pm [q_o]_{d+1} \neq 0$.
	The result now follows from \Cref{l:equivmap}.
\end{proof}

With our new set-up,
any point $(p,0)$ will, in some sense, behave like a realisation of the coned graph $G*o$ with the coned vertex ``placed at infinity''.
It also forces a subset of such points (i.e., those where $p$ is flat in the hyperplane $\mathbb{C}^d \times \{0\}$) to act like they are one dimension lower (minus the coned vertex).
In fact,
the map $g$ shares many properties with the map $f_{G,d}$ when we restrict to a certain subset of elements of $Z_{G,d+1} \times \mathbb{C}$.

\begin{lemma}\label{l:gmapflat}
	Let $\tilde{p} \in X_{G,d}$ be a general realisation of $G$.
	Fix $p = \lambda(\tilde{p}) \in Z_{G,d+1}$.
	Then the following properties hold.
	\begin{enumerate}
		\item \label{l:gmapflat1} For each $(q,r) \in g^{-1}(g(p,0))$ we have $r = 0$, $f_{G,d+1}(q)=f_{G,d+1}(p)$ and $[q_v]_{d+1}=0$ for each $v \in V$.
		\item \label{l:gmapflat2} For each $\tilde{q} \in X_{G,d}$ where $f_{G,d}(\tilde{q}) = f_{G,d}(\tilde{p})$,
		we have $g(\lambda(\tilde{q}),0) = g(p,0)$.
		\item \label{l:gmapflat3} $|\tilde{f}_{G,d}^{-1}(f_{G,d}(\tilde{p}))| = |g^{-1}(g(p,0))|$.
		\item \label{l:gmapflat4} For every $(q,r) \in Z_{G,d+1} \times \mathbb{C}$ we have $\rank \diff g(p,0) \geq \rank \diff g(q,r)$.
		\item \label{l:gmapflat5} For every $(q,r) \in g^{-1}(g(p,0))$,
		the left kernels of $\diff g(q,r)$ and $\diff g(p,0)$ are identical.
	\end{enumerate}
\end{lemma}

\begin{proof}
	\ref{l:gmapflat1}:
	As $g(q,r) = g(p,0)$,
	we have that $r=0$ and $f_{G,d+1}(q)=f_{G,d+1}(p)$.
	Hence,
	\begin{align*}
		h_{v_1}(q,0) = h_{v_1}(q,r) = h_{v_1}(p,0) = 0,
	\end{align*}
	with the latter equality holding since $[p_v]_{d+1}=0$ for all $v \in V$.
	Since $h_{v_1}(q,0) = 0$,
	it now follows that $[q_v]_{d+1}=0$ for each $v \in V \setminus \{v_1\}$.
	The equality $[q_{v_1}]_{d+1} = 0$ also holds since $q \in Z_{G,d+1}$.
	
	\ref{l:gmapflat2}:
	This follows immediately from the observation that $f_{G,d}(\tilde{q}) = f_{G,d+1}(\lambda(q))$ and $f_{G,d}(\tilde{p}) = f_{G,d+1}(p)$.
	
	\ref{l:gmapflat3}:
	This follows from points \ref{l:gmapflat1} and \ref{l:gmapflat2} and the observation that $\lambda$ is a bijection between realisations in $X_{G,d}$ and realisations in $Z_{G,d}$ where the $(d+1)$-th coordinate of each vertex is zero.
	
	\ref{l:gmapflat4}:
	For any $(q,r) \in Z_{G,d+1} \times \mathbb{C}$,
	the Jacobian of $g$ at $(q,r)$ will be a $(|E| + |V|) \times ((d+1)|V| + 1)$ matrix of the form
	\begin{align*}
		\diff g(q,r) =		
		\begin{bmatrix}
			\diff f_{G,d+1}(q) & 0_{|E| \times 1} \\
			A & \left(\frac{1}{2}\|q_v\|^2 \right)_{v \in V \setminus \{v_1\}} \\
			0_{1\times (d+1)|V|} & 1
		\end{bmatrix},
	\end{align*}
	where $0_{1 \times (d+1)|V|}$ is the $1 \times (d+1)|V|$ all zeroes matrix and $A$ is the $(|V|-1) \times (d+1)|V|$ matrix where for the $(v,(w,i))$ entry (with $v \in V \setminus \{v_1\}$ and $(w,i) \in V \times \{1,\ldots, d+1\}$) we have
	\begin{align*}
		A_{v, (w,i)} =
		\begin{cases}
			r [q_v]_i &\text{if } w = v, ~ i \leq d\\
			r[q_v]_{d+1} - 1 &\text{if } w = v, ~ i =d+1\\
			0 &\text{otherwise.}
		\end{cases}
	\end{align*}
	Note that the rank of $\diff g(q,r)$ must satisfy the following upper bound:
	\begin{align}\label{eq1:gmapflat}
		\rank \diff g(q,r) \leq \dim (Z_{G,d+1} \times \mathbb{C}) = (d+1)|V*o| - \binom{d+2}{2}.
	\end{align}
	Now we observe the Jacobian of $g$ when $(q,r) = (p,0)$.
	By moving all columns of $\diff g(p,0)$ that correspond to the $(d+1)$-th coefficients of the vertices to the right,
	we obtain the matrix
	\begin{align}\label{eq1.5:gmapflat}
		\begin{bmatrix}
			\diff f_{G,d}(\tilde{p}) & 0_{|E| \times 1} &  0_{|E| \times |V|} \\
			0_{(|V|-1) \times 1} & (\frac{1}{2}\|p_v\|^2)_{v \in V \setminus \{v_1\}} & -  I_{|V|-1}\\
			0_{1\times (d+1)|V|} & 1 &   0_{1 \times |V|}
		\end{bmatrix},
	\end{align}
	where $I_{|V|-1}$ is the $(|V|-1) \times (|V|-1)$ identity matrix.
	Since $\tilde{p}$ is a general realisation of a $d$-rigid graph,
	we have that
	\begin{align}\label{eq2:gmapflat}
		\rank \diff g(p,0) = \rank \diff f_{G,d}(\tilde{p}) + |V| = d|V| - \binom{d+1}{2} + |V| = (d+1)|V*o| - \binom{d+2}{2}.
	\end{align}
	By combining \Cref{eq1:gmapflat,eq2:gmapflat},
	we have that the derivative of $g$ has maximal rank at $(p,0)$.
	
	\ref{l:gmapflat5}:
	It follows from \ref{l:gmapflat1} that $r=0$ and there exists $\tilde{q} \in X_{G,d}$ where $\lambda(\tilde{q}) = q$ and $f_{G,d}(\tilde{q})=f_{G,d}(\tilde{p})$.
	By reformatting the matrix $\diff g(q,0)$ into the same format as the matrix in \Cref{eq1.5:gmapflat},
	we see that the left kernels of $\diff g(q,0)$ and $\diff g(p,0)$ agree if and only if left kernels of $\diff f_{G,d}(\tilde{p})$ and $\diff f_{G,d}(\tilde{q})$ agree.
	Since $\tilde{p}$ is a general realisation,
	its image is a smooth point in the Zariski closure $\ell_d(G)$ of the image of $f_{G,d}$:
	indeed if this was not true then $\tilde{p}$ would be contained in the preimage of a proper algebraic subset of $\ell_d(G)$ under the morphism $f_{G,d}$, contradicting that it is a general realisation.
	It follows from the inverse mapping theorem for holomorphic maps (see for example \cite[Chapter I, Theorem 7.5]{FG02}) that the tangent space of $\ell_d(G)$ at $f_{G,d}(\tilde{p})$ (respectively, $f_{G,d}(\tilde{q})$) is exactly the left kernel of $\diff f_{G,d}(\tilde{p})$ (respectively, $\diff f_{G,d}(\tilde{q})$).
	Since $f_{G,d}(\tilde{p}) = f_{G,d}(\tilde{q})$,
	the left kernels of $\diff f_{G,d}(\tilde{p})$ and $\diff f_{G,d}(\tilde{q})$ are equal.
	This now concludes the proof.
\end{proof}

\Cref{l:gmap,l:gmapflat} indicate that our map $g$ can be used to consider both the $d$-realisation number of $G$ and the $(d+1)$-realisation number of $G*o$ by observing the fibres of either specific or general elements of the domain of $g$ respectively.
Our next lemma allows us to compare the fibre sizes of these two types of elements.

\begin{lemma}\label{l:diffmaprealisations}
	Let $f : \mathbb{C}^m \rightarrow \mathbb{C}^n$ be a holomorphic map and let $x \in \mathbb{C}^n$ be a point in $\mathbb{C}^n$ where:
	\begin{enumerate}
		\item $|f^{-1}(f(x))|< \infty$,
		\item $\rank \diff f (x) \geq \rank \diff f(y)$ for all $y \in \mathbb{C}^m$, and,
		\item the left kernels of $\diff f(x)$ and $\diff f(z)$ are identical for all $z \in f^{-1}(f(x))$.
	\end{enumerate}
	Then there exists an open neighbourhood $U \subset \mathbb{C}^m$ of $x$ (with respect to the metric topology for $\mathbb{C}^n$) where $\rank \diff f (y) = \rank \diff f(x)$ and $| f^{-1}(f(x))| \leq |f^{-1}(f(y))|$ for all $y \in U$.
\end{lemma}

\begin{proof}
	Label the points in $f^{-1}(f(x))$ as $x_1,\ldots, x_k$,
	with $x_1 = x$.
	As a consequence of the inverse mapping theorem for holomorphic maps (see for example \cite[Chapter I, Theorem 7.5]{FG02}),
	pairwise disjoint open sets $U_1,\ldots , U_k \subset \mathbb{C}^m$ and smooth manifolds $V_1,\ldots,V_k \subset \mathbb{C}^n$ so that for each $j \in \{1,\ldots,k\}$ the following properties hold:
	(i) $x_j \in U_j$,
	(ii) the set $V_j$ is an open neighbourhood of $f(x_j)$ in the image of $f$,
	(iii) the tangent space of $V_j$ is the left kernel of $\diff f(x_j)$,
	and 
	(iv) the restricted map $f|_{U_j}^{V_j}$ is biholomorphic (i.e., invertible with holomorphic inverse).
	Since the left kernel of each Jacobian $\diff f(x_j)$ is identical,
	we may choose our sets so that $V_1 = \ldots = V_k$.
	Choose any $y \in U_1$.
	Since $f|_{U_1}^{V_1}$ is biholomorphic,
	$\rank \diff f(y) = \rank \diff f(x)$.
	For each $j \in \{1,\ldots,k\}$,
	we observe that, since $V_j=V_1$,
	there exists a unique point $y_j \in U_j$ so that $f(y_j)=f(y)$.
	As the sets $U_1,\ldots , U_k$ are pairwise disjoint,
	$| f^{-1}(f(y))| \geq k$.
\end{proof}

We are now ready to use the map $g$ to prove our second main theorem of the paper.

\begin{proof}[Proof of \Cref{t:sphereplane}]
	By \Cref{t:conecount},
	it suffices to prove that $c_d(G) \leq c_{d+1}(G*o)$.
	By \Cref{t:cone},
	it also suffices for us to prove the specific case where $G$ is $d$-rigid (and hence $G*o$ is $(d+1)$-rigid) with at least $d+1$ vertices.
	All notation we now use is in line with our prior notation for the section.
	
	Choose a general realisation $\tilde{p} \in X_{G,d}$ of $G$ and define $p = \lambda(\tilde{p}) \in Z_{G,d+1}$.
	It follows from \Cref{l:gmapflat}\ref{l:gmapflat4}, \Cref{l:gmapflat}\ref{l:gmapflat5} and \Cref{l:diffmaprealisations} that there exists an open neighbourhood
	$U \subset Z_{G,d+1} \times \mathbb{C}$ of $(p,0)$ (with respect to the metric topology) where
	for all $(p',r') \in U$ we have $| g^{-1} (g(p,0)) | \leq | g^{-1}(g(p',r'))|$.
	Since the map $\phi$ is an injective open map with dense image,
	it follows that there exists a general realisation $q \in X_{G*o,d+1}'$ where $\phi(q) \in U$.
	Hence, by \Cref{l:gmap}, \Cref{l:gmapflat}\ref{l:gmapflat3} and \Cref{l:xgd},
	we have that
	\begin{align*}
		2^{d+1} c_d(G) &= 2|\tilde{f}_{G,d}^{-1}(f_{G,d}(\tilde{p}))|
		\\
		&= 2|g^{-1}(g(p,0))|\leq 2|g^{-1}(g(\phi(q)))| \\
		&= |\tilde{f}_{G*o,d+1}^{-1}(f_{G*o,d+1}(q))| \\
		&= 2^{d+1} c_{d+1}(G*o).
	\end{align*}
	Thus $c_d(G) \leq c_{d+1}(G*o)$ as required.
\end{proof}

\begin{remark}\label{rem:alginfo}
	The map $p \mapsto |C_d(G,p)|$ has two important properties:
	(i) it is constant over the set of general realisations,
	and (ii) it is locally minimized at every regular realisation of $G$.
	Theoretically,
	any algebraic property that satisfies points (i) and (ii) will be amenable to the techniques given in this section.
\end{remark}

\section{How and when do spherical and planar realisation numbers differ?}\label{sec:compute}

In this section we compare the $d$-realisation number and spherical $d$-realisation numbers of graphs using computational means.
Since we are utilising computational methods,
we will, for the most part, restrict to results for minimally $d$-rigid graphs.
The reason for this is two-fold:
(i) all known deterministic algorithms for computing realisation numbers require that the graph is minimally 2-rigid,
and (ii) Gr\"{o}bner basis computational methods are too expensive even for comparably small graphs.

A common operation we use in this section is the \emph{($d$-dimensional) 0-extensions}.
This is the graph operation where a new vertex is added to a graph and is set to be adjacent to exactly $d$ vertices of the original graph.
It is well-known that, given a graph $G$ and a graph $G'$ formed from $G$ by a $d$-dimensional 0-extension, the graph $G'$ is (minimally) $d$-rigid if and only if $G$ is (minimally) $d$-rigid (see for example \cite[Proposition 5.1]{TayWhite85}).
As can be seen by the following result,
this operation also behaves very well with realisation numbers.
(See \Cref{sec:0ext} for a proof of this result.)

\begin{lemma}\label{l:0ext}
	If $G'$ is a 0-extension of a graph $G$,
	then $c_d(G') = 2c_d(G)$ and $c_d^*(G') = 2c^*_d(G)$.
\end{lemma}

In what follows we use the combinatorial algorithms from \cite{PlaneCount,SphereCount} respectively their implementations \cite{ZenodoAlg,SphereAlg,RigiComp}.
Note again that the results from these computations are the double of the realisation numbers of what we consider in this paper.
We include computations for all minimally rigid graphs with at most 12 vertices. The respective numbers of realisations for the plane can be found in \cite{ZenodoData}.

\subsection{Ratio of spherical to planar realisation numbers}

In this subsection we evaluate the ratio of spherical $d$-realisation number to $d$-realisation number for minimally $d$-rigid graphs.
We begin by first proving that the minimum possible ratio is, due to our previous results, not an especially interesting value to look into.

\begin{proposition}\label{p:asymmin}
	For any pair of positive integers $n > d$,
	\begin{align*}
		\min \{ c^*_d(G) /c_d(G) : G \text{ is minimally $d$-rigid with $n$ vertices} \} = 1.
	\end{align*}
\end{proposition}

\begin{proof}
	By \Cref{t:sphereplane},
	the minimum for any given value of $n$ is at least 1.
	Define $G_{d+1}$ to be the complete graph with $d+1$ vertices.
	Then $c_d(G_{d+1}) = c^*_d(G_{d+1}) =1$.
	Now construct a sequence $G_{d+1}, G_{d+2},\ldots$ of minimally $d$-rigid graphs,
	whereby $G_{n+1}$ is formed from $G_n$ by a 0-extension.
	It now follows from \Cref{l:0ext} that for each $n >d$ we have $c_d(G_{n}) = c^*_d(G_{n}) = 2^{n - d - 1}$.
\end{proof}

We now turn from the minimum of the ratio to its maximum.
For any positive integers $n > d$,
define
\begin{align*}
	\theta_d(n) := \max \left\{ c^*_d(G) /c_d(G) : G \text{ is minimally $d$-rigid with $n$ vertices} \right\}.
\end{align*}
Here things seem to be possibly more interesting.
For example,
it follows from \Cref{fig:threeprismplane,fig:threeprismsphere} that $\theta_2(6) \geq 16/12 = 4/3$.
In fact we have $\theta_2(6) = 4/3$,
since every other minimally 2-rigid graph $G$ with 6 vertices has $c_2^*(G)=c_2(G) = 8$.

Our first easy observation about $\theta_d(n)$ is that it is increasing:
this is an immediate corollary of \Cref{l:0ext}.
Our next (slightly less easy) observation about $\theta_d(n)$ is that it is bounded above by an exponential function.

\begin{proposition}\label{p:asym}
	For each positive integer $d$,
	there exists a constant $\alpha \geq 1$ such that $\theta_d(n) = O(\alpha^n)$.
\end{proposition}

\begin{proof}
	As shown in \cite{BartzosUpperBound},
	there exists a constant $\alpha>1$ such that
	\begin{align*}
		\max \left\{ c^*_d(G) : G \text{ is minimally $d$-rigid with $n$ vertices} \right\} = O(\alpha^n).
	\end{align*}
	Since $c_d(G) \geq 1$ for any minimally $d$-rigid graph $G$ with at least $d+1$ vertices,
	it follows that $\theta_d(G) = O(\alpha^n)$ also.
\end{proof}

Following from \Cref{p:asym},
we define $\alpha_d$ to be the infimum of all values $\alpha \geq 1$ such that $\theta_d(n) = O(\alpha^n)$.
By \Cref{p:asymmin},
we have $\alpha_d \geq 1$ for any positive integer $d$.
An immediate corollary to \Cref{p:1dsame} is that $\alpha_1=1$.
To prove that $\alpha_d >1$ for any $d \geq 2$, it in fact suffices to prove that $\theta_d(n) >1$ for some positive integer $n$.

\begin{proposition}\label{p:lowerasym}
	Suppose that there exists a minimally $d$-rigid graph $G$ where $c^*_d(G)/c_d(G) >1$.
	Then $\alpha_d >1$.
\end{proposition}

\begin{proof}
	Fix $H$ to be any minimally $d$-rigid graph where $c^*_d(H)/c_d(H) >1$.
	If $H$ has $d$ vertices or less then $c^*_d(H)=c_d(H)=1$,
	thus $H$ has at least $d+1$ vertices and at least one edge.
	Choose any edge $v_0 v_1$ of $H$.
	Given $H_1 = H$,
	we inductively define the graphs $H_1, \ldots, H_d$ by constructing $H_{j+1}$ from $H_j$ by adding a new vertex $v_{j+1}$ adjacent to $v_0,\ldots,v_j$ and $d-j-1$ other vertices.
	Now fix $G := H_d$.
	Since $G$ is formed from $H$ by a sequence of $d$-dimensional 0-extensions,
	it is also minimally $d$-rigid.
	By \Cref{l:0ext} we have that
	\begin{align*}
		\frac{c_d^*(G)}{c_d(G)} = \frac{2^{d-1}c_d^*(H)}{2^{d-1}c_d(H)} = \frac{c_d^*(H)}{c_d(H)} >1.
	\end{align*}
	We now fix $S:=\{v_0,\ldots,v_d\}$ to be the constructed clique in $G$.
	
	Fix $v$ and $e$ to be the number of vertices and edges of $G$ respectively.
	We now construct for each positive integer $k$ the $n=(v-d)k+d$ vertex graph $G_k$ by gluing $k$ copies of $G$ at the clique $S$;
	see \Cref{fig:fan} for an example of $G_4$ when $G$ is the minimally 2-rigid 3-prism.
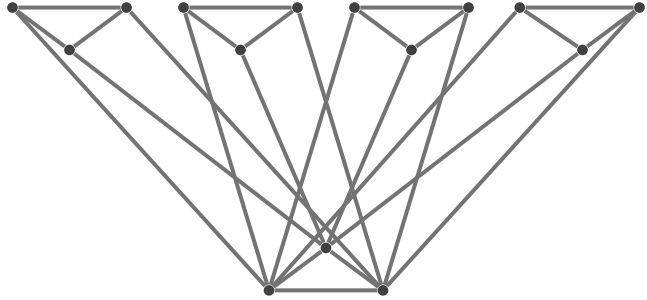
\begin{figure}[ht]
	\centering
		\begin{tikzpicture}[scale=0.75]
			\node[vertex] (a) at (-1,0) {};
			\node[vertex] (b) at (1,0) {};
			\node[vertex] (c) at (0,0.75) {};
			\node[vertex] (d) at (-5.5,5) {};
			\node[vertex] (e) at (-3.5,5) {};
			\node[vertex] (f) at (-4.5,4.25) {};
			\node[vertex] (g) at (-2.5,5) {};
			\node[vertex] (h) at (-0.5,5) {};
			\node[vertex] (i) at (-1.5,4.25) {};
			\node[vertex] (j) at (0.5,5) {};
			\node[vertex] (k) at (2.5,5) {};
			\node[vertex] (l) at (1.5,4.25) {};
			\node[vertex] (m) at (3.4,5) {};
			\node[vertex] (n) at (5.5,5) {};
			\node[vertex] (o) at (4.5,4.25) {};
			\draw[edge] (a)edge(b) (b)edge(c) (c)edge(a);
			\draw[edge] (d)edge(e) (e)edge(f) (f)edge(d);
			\draw[edge] (d)edge(a) (e)edge(b) (f)edge(c);
			\draw[edge] (g)edge(h) (h)edge(i) (i)edge(g);
			\draw[edge] (g)edge(a) (h)edge(b) (i)edge(c);
			\draw[edge] (j)edge(k) (k)edge(l)	(l)edge(j);
			\draw[edge] (j)edge(a) (k)edge(b) (l)edge(c);
			\draw[edge] (m)edge(n) (n)edge(o) (o)edge(m);
			\draw[edge] (m)edge(a) (n)edge(b) (o)edge(c);
		\end{tikzpicture}
	\caption{Gluing four copies of the 3-prism at a common triangle subgraph.}
	\label{fig:fan}
\end{figure}

	We first claim that each graph $G_k$ is minimally $d$-rigid.
	It is relatively intuitive that $G_k$ is $d$-rigid (for a rigorous proof of this statement see \cite[Theorem 2.5.2]{GraverServatius}).
	Since each graph $G_k$ has $k(v -d-1) + d+1$ vertices and
	\begin{align*}
		k\left(e- \binom{d+1}{2} \right) +\binom{d+1}{2} 
		&= k\left(\left(dv - \binom{d+1}{2}\right) - \binom{d+1}{2} \right) +\binom{d+1}{2}  \\
		&= d(k(v -d-1) + d+1) - \binom{d+1}{2} 
	\end{align*}
	edges,
	$G_k$ is also minimally $d$-rigid.
	
	We next claim that $c_d(G_k) = c_d(G)^k$ and $c^*_d(G_k) = c^*_d(G)^k$.
	In lieu of a technical proof, we sketch a proof of this claim as follows.
	Choose a general realisation $p$ of $G$ in $\mathbb{C}^d$ (respectively, $\mathbb{S}^d$).
	If we fix the vertices $v_0,\ldots,v_d$ and count the number of equivalent realisations of $(G,p)$,
	then we see that $(G,p)$ has $c_d(G)$ (respectively, $c^*_d(G)$) such equivalent realisations.
	Hence, each time we glue another copy of $G$ to go from $G_i$ to $G_{i+1}$,
	we must scale the number of realisations by $c_d(G)$ (respectively, $c^*_d(G)$).
	
	Given the graph $G_k$ has $n = k(v -d-1) + d+1$ vertices,
	we see that
	\begin{align*}
		\frac{c^*_d(G_k)}{c_d(G_k)} = \left( \frac{c^*_d(G)}{c_d(G)} \right)^k = \left( \frac{c^*_d(G)}{c_d(G)} \right)^{\frac{n-d-1}{v-d-1}} = \left( \frac{c^*_d(G)}{c_d(G)} \right)^{\frac{-d-1}{v-d-1}} \left( \left( \frac{c^*_d(G)}{c_d(G)} \right)^{\frac{1}{v-d-1}} \right)^n.
	\end{align*}
	Hence $\alpha_d \geq \left( c^*_d(G)/c_d(G)\right)^{\frac{1}{v-d-1}}  >1$.
\end{proof}

\begin{corollary}\label{c:asym12}
	$(4/3)^{3/8} \leq \alpha_2 \leq 2 \cdot 3^{1/2}$.
\end{corollary}

\begin{proof}
	The upper bound for $\alpha_2$ follows from the method employed in \Cref{p:asym}
	with the upper bound for $c_2^*$ over all $n$ vertex minimally 2-rigid graphs being given by 
	\begin{align*}
		8 \cdot 3^{-7/2}\cdot (2 \cdot 3^{1/2})^{n-2};
	\end{align*}
	see \cite[Theorem 18]{Bartzos2023} (remember that our $d$-realisation number is exactly half of the defined $d$-realisation number used in \cite{Bartzos2023}).

	It follows from the proof of \Cref{p:lowerasym} that we can maximise our lower bound for $\alpha_2$ by searching for minimally 2-rigid graphs $G$ with $v$ vertices that contain a triangle where the value $\left( c^*_2(G)/c_2(G)\right)^{\frac{1}{v-3}} $ is high.
	In \Cref{tab:lowerboundgraphs} we have collected some examples where this value is high.
	The name of each graph comes from an integer representation of its adjacency matrix,
	where we take the entries of the upper triangular part (since we always have loop-free graphs) and consider the sequence of row-wise entries as binary digits.
	For example, the triangle graph $K_3$ can be written as $(1,1,1)_2\hat=7$,
	and the 3-prism (\Cref{fig:threeprismplane}) can be written as $(1,1,1,1,0,1,1,1,0,1,1,0,0)_2\hat=7916$. See \cite{LowerBounds} or \cite{ZenodoData} for more details.
	The graphs in the table are those that achive the highest value for $(c_2^*/c_2)^{1/(v-3)}$ with the respective number of vertices.
\begin{table}[ht]
  \centering
  \begin{tabular}{lllllll}
		\toprule
		$v$ & $G$                  & $c_2$ & $c_2^*$ & $c_2^*/c_2$        & $(c_2^*/c_2)^{1/(v-3)}$ \\\midrule
		6   & 7916                 & 12    & 16      & \bnr{1.3333}{colG!50!white} & \bnrsm{1.10064}{20}{colB!50!white}     \\
		7   & 1256267              & 24    & 32      & \bnr{1.3333}{colG!50!white} & \bnrsm{1.07457}{20}{colB!50!white}     \\
		8   & 170957470            & 64    & 96      & \bnr{1.5}{colG!50!white}    & \bnrsm{1.08447}{20}{colB!50!white}     \\
		9   & 2993854888           & 160   & 288     & \bnr{1.8}{colG!50!white}    & \bnrsm{1.10292}{20}{colB!50!white}     \\
		10  & 4847160401729        & 400   & 768     & \bnr{1.92}{colG!50!white}   & \bnrsm{1.09767}{20}{colB!50!white}     \\
		11  & 5366995734673421     & 864   & 2048    & \bnr{2.3704}{colG!50!white} & \bnrsm{1.11391}{20}{colB!50!white}     \\
		12  & 37615476241376327552 & 2016  & 5120    & \bnr{2.5397}{colG!50!white} & \bnrsm{1.10911}{20}{colB!50!white}     \\\bottomrule
  \end{tabular}
  \caption{Graphs and the ratios we obtain from their number of realisations.}
  \label{tab:lowerboundgraphs}
\end{table}

	Fix $G$ to be the 11 vertex graph 5366995734673421, pictured in \Cref{fig:bestforfan}.
	From observation of the table, $G$ has the highest value for $\left( c^*_2/c_2\right)^{\frac{1}{v-3}} $ at $(2048/864)^{\frac{1}{11-3}} \approx 1.11391$.
	Hence, $\alpha_2 \geq (4/3)^{3/8} = (2048/864)^{1/8}$ as required.
	(Note that there can be graphs with more vertices which give a better bound but have not been computed yet.)
\end{proof}

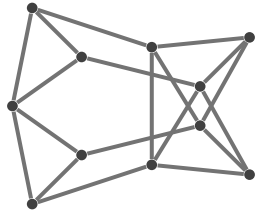
\begin{figure}[ht]
  \centering
  \begin{tikzpicture}[scale=1.3]
		\node[vertex] (1) at (0.7, -0.5) {};
		\node[vertex] (2) at (0.7, 0.5) {};
		\node[vertex] (3) at (0.2, 1) {};
		\node[vertex] (4) at (0.2, -1) {};
		\node[vertex] (5) at (2.4, 0.7) {};
		\node[vertex] (6) at (2.4, -0.7) {};
		\node[vertex] (7) at (0., 0.) {};
		\node[vertex] (8) at (1.9, -0.2) {};
		\node[vertex] (9) at (1.9, 0.2) {};
		\node[vertex] (10) at (1.41, -0.6) {};
		\node[vertex] (11) at (1.41, 0.6) {};
		\draw[edge] (1)edge(4) (1)edge(7) (1)edge(8) (2)edge(3) (2)edge(7)
		(2)edge(9) (3)edge(7) (3)edge(11) (4)edge(7) (4)edge(10) (5)edge(8)
		(5)edge(9) (5)edge(11) (6)edge(8) (6)edge(9) (6)edge(10) (8)edge(11)
		(9)edge(10) (10)edge(11);
	\end{tikzpicture}
	\caption{A graph (known as 5366995734673421) with 864 realisations in the plane and 2048 on the sphere which gives a bound for the triangle-fan of $(2048/864)^{1/8} = (4/3)^{3/8}$.}
	\label{fig:bestforfan}
\end{figure}

By rounding the lower bound down and the upper bound up, \Cref{c:asym12} informs us that $1.1139 \leq \alpha_2 \leq 3.4642$.
It is conjectured in \cite{LowerBounds} that $c_2(G) \geq 2^{|V|-3}$.
If this conjecture is true, we could immediately improve our upper bound to be roughly $1.7321$.

We conclude the subsection by making the following conjecture.

\begin{conjecture}
	For every $d>1$ we have $\alpha_d >1$.
\end{conjecture}

As can be seen from \Cref{p:lowerasym}, it suffices for us to find a single minimally $d$-rigid graph in each dimension $d>2$ where $c^*_d(G) >c_d(G)$.
It follows from \Cref{cor:conecount2} that the problem cannot be solved by merely finding a suitable graph for one dimension and then coning to obtain similar suitable graphs in higher dimensions.
Since there is no deterministic algorithm for higher dimensions,
we could only use polynomial system solving tools (like for instance Gröbner basis) with random edge lengths.
As well as only being able to give a probabilistic answer,
this method is computationally infeasible and can only be done for small numbers of vertices.
For the graphs we were able to compute using this method we saw that $c^*_d(G)=c_d(G)$ always, 
but this is most likely because they were too small for the values to begin differing.

\subsection{Exploring data sets}

The next computational question we approach is the following:
how often do the spherical $d$-realisation number and the $d$-realisation number agree for minimally $d$-rigid graphs?
Interestingly these two numbers seem to differ more than they agree.
\Cref{tab:differences} shows the amount of minimally 2-rigid graphs 
for which the two realisation numbers differ.
We solely consider those minimally 2-rigid graphs with minimal degree 3 since the removal of a degree 2 vertex alters both the spherical 2-realisation number and the 2-realisation number by a factor of $1/2$.

\begin{figure}[ht]
  \centering
  \begin{tikzpicture}
    \newcommand{\s}{10}
    \foreach \nn/\gt/\ge/\gi/\e/\i [evaluate=\nn as \n using \nn/2] in {6/2/1/1/0.5/0.5,7/4/1/3/0.25/0.75,8/32/7/25/0.21875/0.78125,9/264/42/222/0.159091/0.840909,10/3189/330/2859/0.103481/0.896519,11/46677/3063/43614/0.0656212/0.934379,12/813875/32855/781020/0.0403686/0.959631}
    {
        \node[labelsty,colfg,anchor=east] at (0,-\n) {$n=\nn$};
        \fill[colB] (0,-\n-0.2) rectangle (\s*\e,-\n+0.2);
        \fill[colR] (\s*\e,-\n-0.2) rectangle (\s*\e+\s*\i,-\n+0.2);
        \node[slabelsty,colbg,anchor=west] at (0,-\n) {\ge};
        \node[slabelsty,colbg,anchor=east] at (\s*\e+\s*\i,-\n) {\gi};
        \node[slabelsty,colfg,anchor=west] at (\s*\e+\s*\i,-\n) {\gt};
    }
    \node[labelsty,colB,anchor=west] at (0,-2.5) {\# graphs with $c_2=c_2^*$};
    \node[labelsty,colR,anchor=east] at (\s,-2.5) {\# graphs with $c_2\neq c_2^*$};
  \end{tikzpicture}

  \caption{The number of minimally 2-rigid graphs (up to isomorphism) with $n$ vertices and minimum degree 3 for which $c_2(G)$ is the same as/different from $c_2^*(G)$.}
  \label{tab:differences}
\end{figure}

In light of this experimental data,
we believe that the following conjecture is true.

\begin{conjecture}
	Let $d$ be an integer greater than 1.
	Let $A_{n,d}$ (respectively, $B_{n,d}$) be the set of all minimally $d$-rigid graphs (up to isomorphism) with $n$ vertices and minimum degree $2d-1$ where $c^*_d(G)=c_d(G)$ (respectively, $c^*_d(G)\neq c_d(G)$).
	Then $|A_{n,d}|/ |B_{n,d}| \rightarrow 0$ as $n \rightarrow \infty$.
\end{conjecture}

After discussing the number of graphs for which the realisation counts differ, we are also interested in how they differ.
\Cref{fig:distribution12} shows this relation for all 813875 minimally 2-rigid graphs with 12 vertices and minimum degree 3.
We observe that there are 9916 different pairs $(c_2(G),c^*_2(G))$ (ignoring repetition) that occur.
Only five of those pairs have equal coordinates (i.e., $c_2(G)=c^*_2(G)$):
$(512,512)$, $(768,768)$, $(869,869)$, $(960,960)$ and $(1024,1024)$.
The majority (30789) of graphs for which $c_2(G)=c^*_2(G)$ have 512 realisations.
In the figure we colour coded the amount of occurrences of the pairs $(c_2(G),c^*_2(G))$,
with blue indicating few occurrences and red implying many occurrences.
The most frequent pair is $(768,1024)$, which occurs for 76025 graphs.
The pair with the largest difference is $(2496, 6144)$ which is obtained by a single graph.
Including the graph indicated in the last row of \Cref{tab:lowerboundgraphs}, there are 20 graphs with the pair $(2016, 5120)$, which gives the largest ratio.

In \Cref{fig:ratios} we analyse more deeply the ratios $c^*_2(G)/c_2(G)$ for minimally 2-rigid graphs with minimal degree 3.
We can see, confirming \Cref{p:asymmin}, that the minimum achievable ratio is 1.
The maximum ratio in the range is achieved by some graphs with 12 vertices, one of which is given in \Cref{tab:lowerboundgraphs}.
Interestingly, the median ratio varies between 1.39 and 1.4 depending on vertex number.
Although the maximum achievable ratio is increasing exponentially (see \Cref{c:asym12}),
the range of the quartiles do not seem to vary much as the number of vertices is increased.
Note, however, that the size of the graphs considered is rather limited.

\begin{figure}[ht]
 \centering
 \includegraphics[width=6cm]{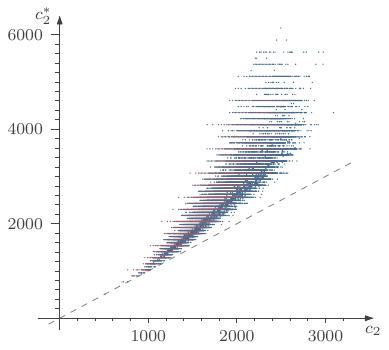}
 \caption{The distribution of pairs $(c_2(G),c^*_2(G))$ for all minimally 2-rigid graphs $G$ with 12 vertices and minimum degree three. We use colour to represent the number of occurrences of a given pair: the more often a point appears in the list of all possible pairs, the more red its coordinate representation is in our plot. All points lie above the dashed line representing the equality $c_2(G)=c^*_2(G)$, showing the main result of the paper.}
 \label{fig:distribution12}
\end{figure}

\begin{figure}[ht]
 \centering
 \begin{tikzpicture}[yscale=1.2,boxplot/.style={draw=colB,fill=colB!50!white,rounded corners=0.5pt}]
  \foreach \n [count=\i] in {8,...,12}
  {
    \draw[aline] (\i,0.9) node[below,alabelsty] {$\n$};
  }
  \foreach \y [count=\i] in {1,1.5,...,2.5}
  {
    \draw[bline] (0,\y)--(-0.1,\y) node[left,alabelsty] {\y};
    \draw[bline] (0,\y)--(6,\y);
  }
  \node[rotate=-90,alabelsty] at (6.25,1.75) {ratio};
  \node[alabelsty] at (3,2.75) {$|V|$};
  \foreach \min/\q/\median/\p/\max [count=\i] in {1./1.06667/1.14286/1.33333/1.5, 1./1.09091/1.17647/1.33333/1.8, 1./1.14286/1.26316/1.33333/1.92, 1./1.14286/1.33333/1.41176/2.37037, 1./1.21212/1.33333/1.48837/2.53968}
  {
    \draw[boxplot] (\i,\min)--(\i,\q) (\i,\p)--(\i,\max) (\i-0.125,\min)--(\i+0.125,\min) (\i-0.125,\max)--(\i+0.125,\max);
    \draw[boxplot] (\i-0.25,\q) rectangle (\i+0.25,\p);
    \draw[boxplot] (\i-0.25,\median)--(\i+0.25,\median);
  }
 \end{tikzpicture}
 \caption{A box plot of the distribution of ratios $c^*_2(G)/c_2(G)$ for minimally 2-rigid graphs with between 8 to 12 vertices and minimal degree 3.}
 \label{fig:ratios}
\end{figure}
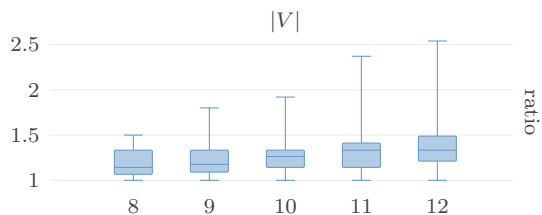

\addcontentsline{toc}{section}{Acknowledgments}
\section*{Acknowledgements}

The authors would like to thank Matteo Gallet, Dániel Garamvölgyi and Jan Legersk\'{y} for their helpful discussions and valuable feedback.

Both authors were supported by the Austrian Science Fund (FWF): P31888.
SD was also supported by the Heilbronn Institute for Mathematical Research.

\phantomsection
\addcontentsline{toc}{section}{References}
\bibliographystyle{plainurl}
\bibliography{References}

\appendix

\section{Dominant morphisms}
\label{sec:dom}

Dominant morphisms can be defined in a variety of different but equivalent ways.

\begin{theorem}[{\cite[Section AG, Theorem 17.3]{borel91}}]\label{borel91}
	Let $X \subset \mathbb{C}^m$ be an algebraic set and $Y \subset \mathbb{C}^n$ be a variety.
	Then the following are equivalent for any morphism $f :X \rightarrow Y$:
	\begin{enumerate}
		\item $f$ is dominant.
		\item For some irreducible component $X'$ of $X$, there exists a point $x \in X'$ such that $x$ is a non-singular point of $X'$ and $\rank \diff f(x) = \dim Y$.
		\item There exists a Zariski open subset $U \subset X$ where for each $x \in U$, $x$ is a non-singular point of $X$ and $\rank \diff f(x) = \dim Y$.
	\end{enumerate}
\end{theorem}

It follows immediately from \Cref{borel91} that for any varieties $X,Y$,
the existence of a dominant morphism from $X$ to $Y$ implies $\dim X \geq \dim Y$.
As can be seen by \Cref{t:deg},
more powerful statements can be attained relating to $f$ if $\dim X=\dim Y$.
To prove \Cref{t:deg}, we require the following two results.

\begin{corollary}[{\cite[Section 8, Corollary 1]{mumford}}]\label{mumford}
	Let $X \subset \mathbb{C}^m$ and $Y \subset \mathbb{C}^n$ be varieties and $f:X \rightarrow Y$ be a dominant morphism.
	Then there exists a non-empty Zariski open subset $U \subset Y$ such that $U \subset f(X)$, and for every $y \in U$, every irreducible component of the algebraic set $f^{-1}(y)$ has dimension $\dim X - \dim Y$.
\end{corollary}

\begin{corollary}[{\cite[Proposition 7.16]{Harris}}]\label{c:harris}
	Let $X \subset \mathbb{C}^m$ and $Y \subset \mathbb{C}^n$ be varieties and $f:X \rightarrow Y$ be a dominant morphism.
	Suppose that there exists a non-empty Zariski open subset $U \subset Y$ where $|f^{-1}(y)|<\infty$ for every $y \in Y$.
	Then there exists a $k \in \mathbb{N}$ and a non-empty Zariski open subset $U' \subset U$ where $|f^{-1}(y)| = k$ for every $y \in U'$.
\end{corollary}

\begin{proof}[Proof of \Cref{t:deg}]
	It is immediate that \ref{t:deg3} implies \ref{t:deg2}.
	Fix $U \subset Y$ to be the Zariski open set from \Cref{mumford}.
	An algebraic set is a finite set if and only if it is zero dimensional.
	Hence, \ref{t:deg1} and \ref{t:deg2} are equivalent.
	Suppose that $\dim X=\dim Y$.
	Fix $U' \subset U$ to be the non-empty Zariski open subset from \Cref{c:harris}.
	Define the set
	\begin{align*}
		C := \{x \in X : x \text{ is singular or } \rank \diff f(x) <n\}.
	\end{align*}
	By \Cref{borel91},
	$C$ is a proper Zariski closed subset of $X$.
	As $\dim C < \dim X = \dim Y$,
	the set $f(C)$ is not dense in $Y$.
	It now follows that the complement of the Zariski closure of $f(C)$ in $U$ is a non-empty Zariski open subset of $Y$.
	Hence, \ref{t:deg1} implies \ref{t:deg3},
	concluding the proof.
\end{proof}

\section{The effect of 0-extensions on realisation numbers}
\label{sec:0ext}

In this section we prove \Cref{l:0ext}.
The specific case where $d=2$ was originally proven in \cite{BorceaStreinu}.
We restrict to the non-spherical case throughout this section since the proof is almost identical.

\begin{lemma}\label{l:0extvec}
	Let $p_0,\ldots,p_d \in \mathbb{C}^d$ be affinely independent points where $[p_j]_k = 0$ for all $1 \leq j \leq k \leq d$.
	Then there exists exactly one point $x \in \mathbb{C}^d\setminus \{p_0\}$ which is a solution to the set of equations
	\begin{align}\label{eq:0ext1}
		\| x - p_j \|^2 = \| p_0 - p_j \|^2, \qquad j \in \{1,\ldots,d\}.
 	\end{align}
\end{lemma}

\begin{proof}
	First note that we must have that $[p_0]_d \neq 0$ for $p_0,\ldots,p_d \in \mathbb{C}^d$ to be affinely independent.
	Let $x \in \mathbb{C}^d$ be a solution to the equations in \cref{eq:0ext1}.
	Then the set of points $\{x,p_1,\ldots,p_d\}$ have the same set of pairwise distances between them as the set of points $\{p_0,p_1,\ldots,p_d\}$.
	Hence, there exists a map $M \in O(d,\mathbb{C})$, where $Mp_j=p_j$ for each $j \in \{1,\ldots,d\}$ and $Mp_0 = x$;
	see \cite[Section 10]{Gortler2014} for more details.
	Since $M$ is invariant over $p_1,\ldots,p_d$, it follows that $M$ is either the identity matrix or $M$ is the diagonal matrix with $M_{jj}=1$ for each $j<d$ and $M_{dd}=-1$.
	The result now follows immediately.
\end{proof}

\begin{proof}[Proof of \Cref{l:0ext}]
	For the linear spaces $X_{G,d}$ and $X_{G',d}$,
	fix the vertices $v_1,\ldots,v_d \in V$ to be the vertices adjacent to the new vertex $u$ that is appended during the 0-extension operation.	
	Fix a general realisation $p \in X_{G,d}$ of $G$.
	For each $q \in \tilde{f}_{G,d}^{-1}(f_{G,d}(p))$,
	define the dominant morphism
	\begin{align*}
		f_q : \mathbb{C}^d \rightarrow \mathbb{C}^d, ~ x \mapsto \left(\|x - q_{v_j}\|^2\right)_{j =1}^d
	\end{align*}
	and the non-empty Zariski open set $U_q \subset \mathbb{C}^d$ of points not contained in the affine span of $q_{v_1},\ldots,q_{v_d}$.
	Since each map $f_q$ is dominant,
	it follows from \Cref{mumford} that there exists a non-empty Zariski open set $U \subset \mathbb{C}^d$ such that 
	\begin{align*}
		U \subset \bigcap \left\{ f_q(\mathbb{C}^d) : q \in \tilde{f}_{G,d}^{-1}(f_{G,d}(p)) \right\}.
	\end{align*}
	From this we note that the set
	\begin{align*}
		U' := \bigcap \left\{ U_q \cap f_q^{-1}(U) : q \in \tilde{f}_{G,d}^{-1}(f_{G,d}(p)) \right\}
	\end{align*}
	is a non-empty Zariski open set.
	Hence, there exists a general realisation $p' \in X_{G',d}$ of $G'$ with $p'_v=p_v$ for all $v \in V$ and $p'_u \in U'$.
	By applying \Cref{l:0extvec} to every realisation in $X_{G',d}$ that is equivalent to $p'$,
	we see that
	\begin{align*}
		\left|\tilde{f}_{G',d}^{-1}\left(f_{G',d}(p') \right) \right| = 2\left|\tilde{f}_{G,d}^{-1}\left(f_{G,d}(p) \right) \right|.
	\end{align*}
	The result now follows from \Cref{l:xgd}.
\end{proof}

\end{document}